\documentclass[11pt,letterpaper]{amsart}
\usepackage{amssymb}
\usepackage{amsfonts}
\usepackage{mathrsfs}
\usepackage{amsmath}
\usepackage{graphicx}
\usepackage{hyperref}
\usepackage{float}
\usepackage{epstopdf}
\usepackage{color}
\usepackage{bm}
\usepackage{comment}
\usepackage{soul}

\allowdisplaybreaks
\newtheorem{theorem}{Theorem}[section]
\newtheorem{proposition}[theorem]{Proposition}
\newtheorem{lemma}[theorem]{Lemma}
\newtheorem{corollary}[theorem]{Corollary}

\newtheorem{example}[theorem]{Example}
\newtheorem{definition}[theorem]{Definition}

\def\mcD{\mathcal{D}}
\def\mcE{\mathcal{E}}
\def\mcF{\mathcal{F}}
\def\mfL{\mathfrak{L}}
\def\mfI{\mathfrak{I}}
\def\mfR{\mathfrak{R}}
\def\bw{\bm{w}}
\def\mbP{\mathbb{P}}

\numberwithin{equation}{section}

\begin{document}
\title[loop-erased random paths on resistance spaces]{Scaling limits of loop-erased Markov chains on resistance spaces via a partial loop-erasing procedure}

\author{Shiping Cao}
\address{Department of Mathematics, University of Washington, Seattle 98105, USA}
\email{spcao@uw.edu}
\thanks{}

%    General info

\date{}


\begin{abstract}
We introduce partial loop-erasing operators. We show that by applying a refinement sequence of partial loop-erasing operators to a finite Markov chain, we get a process equivalent to the chronological loop-erased Markov chain. As an application, we construct loop-erased random paths on bounded domains of resistance spaces as the weak limit of the loop erasure of the Markov chains on a sequence of finite sets approximating the space, and the limit is independent of the approximating sequences. The random paths we constructed are simple paths almost surely, and they can be viewed as the loop-erasure of the paths of the diffusion process. Finally, we show that the scaling limit of the loop-erased random walks on the Sierpi\'nski carpet graphs exists, and is equivalent to the loop-erased random paths on the Sierpi\'nksi carpet. 
\end{abstract}
\maketitle

\section{introduction}
In the 1980's, Lawler introduced loop-erased random walks (LERW) in \cite{L1}, and the model later gains much interest due to its connection with uniform spanning trees \cite{MDM,P,W}.

One of the natural questions about LERWs is the existence of their scaling limits. On Euclidean spaces, significant progress has been made over the past few decades. First, as a main result of \cite{L1}, Lawler showed that, in the high dimensional cases ($d\geq 5$), the scaling limit of LERW on $\mathbb{Z}^d$ exists and is the same as the Brownian motion on $\mathbb{R}^d$. A similar result holds in the $d=4$ case \cite{L2,L3,L4} with an additional logarithmic correction in time. The question becomes more delicate for $d=2,3$ cases. On $\mathbb{R}^2$, Lawler, Schramm and Werner \cite{LSW} proved that the scaling limit exists, which is characterized by SLE$_2$ \cite{S}. On $\mathbb{R}^3$, Kozma proved the existence of the scaling limit \cite{Ko}. It was shown later in \cite{SS} by Sapozhnikov and Shiraishi that the limit is almost surely a simple path. Moreover, a time parameterized scaling limit is obtained recently in \cite{LV} on $\mathbb{R}^2$, and in \cite{LS} on $\mathbb{R}^3$. The reader is referred to \cite{Kenyon,M,D.S} for the study of the growth rate of LERW.

In this paper, we focus on establishing the existence of the scaling limit of LERW on some fractal spaces. In fact, in earlier works \cite{H,HM}, K. Hattori and M. Mizuno have studied LERW on Sierpi\'nski gasket graphs, using an `erasing-larger-loops-first' algorithm. Inspired by their idea, in this paper, we will develop a much simpler and more general algorithm. Using the new algorithm, we can naturally define loop-erased random paths (LERP) on a large class of spaces, including many fractals. More precisely, in a proper local resistance space \cite{ki3,ki4,ki5}, we will define LERP as the weak limit in Hausdorff metric of the loop erasure of Markov chains on finite sets approximating the domain. Typical examples of resistance spaces are strongly local regular Dirichlet spaces that have sub-Gaussian heat kernel estimates with the walk dimension \cite{B,BCK,GHL} greater than the Hausdorff dimension (as a consequence of the Sobolev embedding theorem, see Theorem 4.11 of \cite{GHL} for example). The author hopes that these results can be extended to more general settings in the future, in particular, dropping the condition about the pointwise recurrence: every point has positive capacity in a resistance space.

% where the equivalence of the algorithm and LE was tested with the aid of a related work \cite{STW} counting spanning trees on the Sierpi\'nski gasket graphs. In this paper, we will introduce a similar, but more general \vspace{0.15cm}

 \vspace{0.15cm}

The main tool of this paper, as pointed out in the title, is called partial loop erasure (PLE), which is a generalization of loop erasure (LE). When we apply the LE operator to a finite path on a finite set $V$, we erase all the loops in chronological order. On the other hand, the PLE operator associated with $\widetilde{V}\subset V$ is that we only erase loops about points in $\widetilde{V}$. \vspace{0.15cm}

\noindent\textbf{\emph{Loop erasure (LE).}} Let $\bw=(w_0,w_1,\cdots,w_\eta)$ be a finite path on a finite set $V$. First, let $n_0=0$. Then iteratively, for $i\geq 1$, we define
\[n_i=\max\{n_{i-1}\leq n\leq\eta:w_n=w_{n_{i-1}}\}+1,\quad\text{ if }w_{n_{i-1}}\neq w_\eta.\]
We let $n_\iota$ be the last index found by the algorithm (so $w_{n_\iota}=w_\eta$). Define 
\[\mfI(\bw)=(n_0,n_1,\cdots,n_\iota),\quad \mfL(\bw)=(w_{n_0},w_{n_1},\cdots,w_{n_\iota}). \] 
We call $\mfL(\bw)$ the loop erasure (LE) of $\bw$, and call $\mfI(\bw)$ the index set of the LE of $\bw$.

\noindent\textbf{\emph{Partial loop erasure (PLE).}} Let $\widetilde{V}\subset V$ be finite sets. Let $\bw=(w_0,w_1,\cdots,w_\eta)$ be a finite path on $V$. First, let $n_0=0$. Then iteratively, for $i\geq 1$, we define
\[
n_i=\begin{cases}
	\max\{n_{i-1}\leq n\leq\eta:w_n=w_{n_{i-1}}\}+1,&\text{if }w_{n_{i-1}}\in \widetilde{V}\text{ and }w_{n_{i-1}}\neq w_\eta,\\
	n_{i-1}+1,&\text{if }w_{n_{i-1}}\notin \widetilde{V}\text{ and }n_{i-1}\neq \eta.
\end{cases}
\]
We let $n_\iota$ be the last index found by the algorithm (so $w_{n_\iota}=w_\eta$). Define 
\[\mfI_{\widetilde{V}}(\bw)=(n_0,n_1,\cdots,n_\iota),\quad \mfL_{\widetilde{V}}(\bw)=(w_{n_0},w_{n_1},\cdots,w_{n_\iota}). \] 
We call $\mfL_{\widetilde{V}}(\bw)$ the partial loop erasure (PLE) asscoiated with $\widetilde{V}$ of $\bw$, and call $\mfI_{\widetilde{V}}(\bw)$ the index set of the PLE associated with $\widetilde{V}$ of $\bw$.  

In particular, $\mfL_V=\mfL$, and $\mfL_{\emptyset}$ is the identity operator. \vspace{0.15cm}

A refinement sequence of PLE is that we apply PLE a few times, where each time we expand the range $\widetilde{V}$. Surprisingly, although a refinement sequence of PLE can lead to a very different path from the standard LE (see Example \ref{example22}), when we talk about the Markov chain, we will prove that they lead to equivalent processes. More precisely, we consider the random finite path $X|_{[0,\tau_A]}=(X_0,X_1,\cdots,X_{\tau_A})$, where $X_n,n\geq 0$ is a Markov chain on $V$ and $\tau_A=\min\{n\geq 0:X_n\in A\}$ is the entry time. We will prove the following theorem soon in Section 2.

\begin{theorem}\label{th1}
	Let $\big(\Omega,X_n,\mbP_x\big)$ be a discrete time Markov chain (without killing) on a finite set $V$. Let $V_1,V_2,\cdots,V_m$ be a sequence of subsets of $V$ such that
	\[\emptyset\neq V_1\subset V_2\subset \cdots\subset V_m=V.\]
	Then for any $A\subset V$ and $x\in V$ such that $\mbP_x(\tau_A<\infty)=1$, we have 
	\[\mbP_x\big(\mfL(X|_{[0,\tau_A]})\in \bullet\big)=\mbP_x\big(\mfL_{V_m}\circ\cdots\circ\mfL_{V_2}\circ \mfL_{V_1}(X|_{[0,\tau_A]})\in \bullet\big).\]
\end{theorem}

As mentioned earlier, on the Sierpi\'nski gasket graphs a slightly different idea, named `erasing-larger-loops-first', was introduced by K. Hattori and M. Mizuno in \cite{H,HM}, where the equivalence of the algorithm and LE was tested with the aid of a related work \cite{STW} counting spanning trees on the Sierpi\'nski gasket graphs. There hasn't been a satisfying explanation of the phenomenon over the past few years. Our new algorithm doesn't depend on the structure of the underlying graph, and we have a fundamental proof without massive computation.\vspace{0.15cm}

Next, using the tool PLE, we study LERP on resistance spaces (Section 3 and 4) and the scaling limit of LERW on Sierpi\'nski carpets (Section 5). 

We will review important concepts about resistance spaces in Section 3, where readers can find formal explanations about the concepts. For a short explanation, a resistance space (\cite{ki3,ki4,ki5}) is a pair $(K,R)$, where $K$ is a set and $R$ is a metric on $K$, such that on any finite subset $V$ of $K$, $R$ can be viewed as the effective resistance induced by some electrical network on $V$ (see \cite{LP}). $R$ is called the resistance metric for this reason. In this paper, we consider proper local regular resistance spaces, where `proper' means any bounded closed subset is compact so the space is locally compact. Typical examples of such resistance spaces are post critically finite (p.c.f.) self-similar sets \cite{ki1,ki2,ki3} and many Sierpsi\'nski carpet type fractals \cite{BB1,BB4,CQ,KZ}. Also see \cite{C1,CHK} for some other examples.  It is well known that given a Radon measure $\mu$ with full support on a local, regular, locally compact resistance space $(K,R)$, there is an associated diffusion process $(\Omega,X_t,\mathbb{P}_x)$ (\cite{ki3,FOT}).

For each finite $V\subset K$, we let $X^{[V,A]}|_{[0,\xi]}$ be the random finite paths that consists of the points that $X$ hits in $V$ before being killed at $A$ in chronological order. It can also be viewed as the random walk (Markov chain) induced by some electrical network (the trace of the resistance form on $V\cup \{A\}$, where we view $A$ as a single point \cite{ki5}).  $X^{[V,A]}|_{[0,\xi]}$ will approximate the sample paths of the diffusion process as $V$ becomes denser, and we will show that the LE of  $X^{[V,A]}|_{[0,\xi]}$ will converge in distribution with respect to the Hausdorff metric to some random simple (continuous) paths on $(K,R)$. 

More precisely, we will prove the following theorem in Section 4. 

\begin{theorem}[LERP on resistance spaces]\label{th2}
	Let $(K,R)$ be a proper local regular resistance space, let $\mu$ be a Radon measure on $K$ with full support, let $(\Omega,X_t,\mathbb{P}_x)$ be the associated diffusion process, and let $A=K\setminus U$ for some bounded open $U\subsetneq K$. Then, for each $x\in U$, there exists a probability measure $\nu_{x,A}$ on $(\mathcal{K},d_H)$ supported on simple paths connecting $x$ and $A$ such that
	\[\mathbb{P}_x\big(\Im\mfL(X^{[V_m,A]}|_{[0,\xi]})\in \bullet\big)\Rightarrow \nu_{x,A},\text{ as }m\to\infty\]
	on $(\mathcal{K},d_H)$ for any expanding sequence $V_1\subsetneq V_2\subsetneq V_3\subsetneq \cdots$ of finite subsets of $U$ such that $V_*=\bigcup_{m=0}^\infty V_m$ is dense in $U$. 
\end{theorem}

In the statement of the theorem, $\Im$ denotes the image of a map: in other words, we view a path as a compact subset of $K$. $(\mathcal{K},d_H)$ is the space of compact subsets of $(K,R)$ endowed with the Hausdorff metric. See Section 4 for formal definitions.

\begin{figure}[htp]
	\includegraphics[width=4cm]{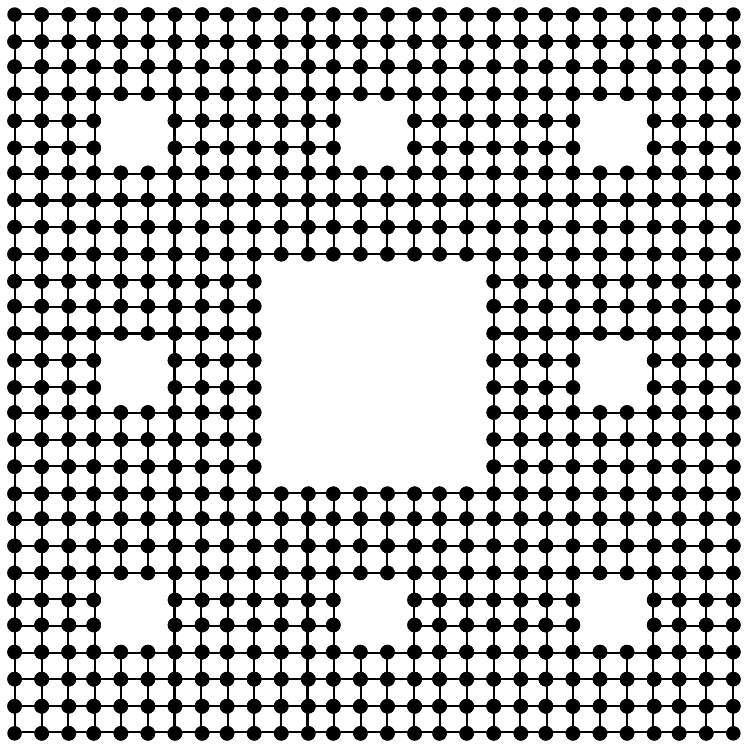}
	\caption{A Sierpi\'nski carpet graph.}\label{fig1}
\end{figure}

Finally, we end the paper with a study of the Sierpi\'nski carpets ($\mathcal{SC}$), which have been a central model in the area of analysis of fractals. Compared with p.c.f. fractals, the ways to construct diffusions on $\mathcal{SC}$ are more complicated \cite{BB1,BB4,CQ,KZ}. Fortunately, in \cite{BBKT}, Barlow, Bass, Kumagai and Teplyaev proved the uniqueness of locally symmetric Dirichlet forms on $\mathcal{SC}$, and based on their result, we can show the existence of the scaling limit of LERW on $\mathcal{SC}$ graphs. This question is also the initial motivation of the paper. See Figure \ref{fig1} for a typical $\mathcal{SC}$ graph.

For simplicity, we focus on a simple case where the paths have starting point $x$ and endpoint $y$, for some $x,y\in V_*$. Here $V_*$ is the union of vertice sets of $\mathcal{SC}$ graphs. The method can be generalized to other cases, for example that the path ends at the boundary of the square (or some sub-intervals of the boundary).  

\begin{theorem}\label{th3}
	Let $K$ be a $\mathcal{SC}$, let $x,y\in V_*$, and let $\nu_{x,y}$ be the distribution of loop-erased random paths defined in Theorem \ref{th2}, which start at $x$ and end at $y$. Then 
	\[
	\mathbb{P}_x^{(m')}\big(\Im\mfL(X^{(m')}|_{[0,\tau_y]})\in \bullet\big)\Rightarrow \nu_{x,y},\quad \text{ as }m'\to\infty,
	\]  
	where the weak convergence is on $(\mathcal{K},d_H)$.
\end{theorem}

We also need to remark here that in this paper we do not consider the time, or the growth exponent of the LERW or LERP. The growth exponent is also of great interest, and worth independent study, on $\mathcal{SC}$ graphs.\vspace{0.15cm}

At the end of the introduction, we have a remark about terminologies. In this paper, a path can be a continuous path or a discrete path. To distinguish between them, we always use the notation $\bw=(w_1,w_2,\cdots)$ for a discrete path, and use $\gamma:[0,s]\to K$ to denote a continuous path. In Section 2, we only care about Markov chains and discrete paths. In Section 3,4,5, $X$ will refer to the diffusion process, which can be viewed as random continuous paths, and we will add some superscript like $ ^{[V,A]}$, $ ^{(m)}$ when we talk about a Markov chain/random walk. 

%In Section 2, we introduce PLE, and prove that applying a refinement sequence of PLE to a Markov chain leads to the same process as the LE of the Markov chain. This section will be self-contained. 

%In Section 3 and 4, we define LERP on resistance spaces. In Section 3, we review definitions about resistance spaces, and talk about useful estimates of hitting times. In Section 4, we show that the PLE of the sample paths of the diffusion process, or the LE of Markov chains associated with the form, will converge weakly to simple paths in Hausdorff metric or the metric over unparameterized paths. Theorem \ref{th1} will play a key role by providing a natural coupling over all the random paths. 

%In Section 5, we consider LERW on $\mathcal{SC}$ graphs. This section will involve deep results about the diffusion processes on $\mathcal{SC}$. However, the theorems about diffusions on $\mathcal{SC}$ are natural, and not knowing the proof of them should not affect the ability of readers to understand our story.\vspace{0.15cm}

\section{A refinement sequence of partial loop erasure}
We focus on the proof of Theorem \ref{th1}, and we will provide all the necessary definitions and proofs. Readers can also read the book \cite{L3} and the surveys \cite{L5,L6} for more background knowledge.

The structure of this section is as follows. In the first part, we briefly introduce notations that we haven't explained in the introduction. In the second subsection, the major part of the section, we prove Theorem \ref{th1}. Finally, at the end of this section, as an easy example, we will briefly discuss how Theorem \ref{th1} is applied to Sierpi\'nski gasket graphs.  
 
\subsection{A review of basic concepts}\label{sec2.1} We explain the concepts and notations that appear in the statement of Theorem \ref{th1}. Throughout this paper, we say a finite path $\bm{w}=(w_1,w_2,\cdots, w_\eta)$ is on $V$ if $w_n\in V,\forall 0\leq n\leq \eta$. Also, with a slight abuse of notation, formulas like $x_0,x_1,\cdots,x_j$ simply means the empty sequence when $j=-1$. \vspace{0.15cm}

First, we briefly review the definition of Markov chains. In this paper, `Markov chain' always means a Markov chain without killing. One can also easily apply the theorem to a \emph{Markov chain with killing} (sub-Markov chain), by viewing the killing time as the entry time $\tau_\triangle$ of some absorbing state $\triangle$.\vspace{0.15cm}

\noindent\textbf{\emph{Markov chain on a finite set.}} Let $V$ be a finite set, and let $P$ be a Markov kernel on $V$, i.e. $P:V\times V\to [0,1]$ and $\sum_{y\in V}P(x,y)=1,\forall x\in V$. Then there is a Markov chain $(\Omega,X_n,\mbP_x)$, where $\Omega$ is the sample space, $X_n:\Omega\to V$ is a random variable for each $n\geq 0$, and $\mathbb{P}_x$ is a probability measure on $\Omega$ for each $x\in V$, such that for any $x,y\in V$,
\[
\mbP_x(X_n=y)=P_n(x,y):=\sum_{y_1,y_2,\cdots,y_{n-1}\in V}P(x,y_1)P(y_1,y_2)\cdots P(y_{n-1},y)
\]
for any $n\geq 1$, and $\mathbb{P}_x(X_0=y)=P_0(x,y):=\delta_{x,y}$. Here $\delta_{x,y}$ is the kronecker delta, i.e. $\delta_{x,y}=0$ if $x\neq y$; $\delta_{x,y}=1$ if $x=y$. We also let $\theta_n:\Omega\to\Omega,n\geq 0$ be the shift mappings, i.e. 
\[
X_n\circ \theta_{n'}=X_{n+n'},\qquad \forall n,n'\geq 0.
\]
For each $A\subset X$, we define the \emph{hitting time} $\mathring{\tau}_A$ and the \emph{entry time} $\tau_A$ as
\[\mathring{\tau}_A:=\min\{n\geq 1:X_n\in A\},\quad \tau_A:=\min\{n\geq 0:X_n\in A\},\]
where we set $\min\emptyset=\infty$. In addition, we let \[\mathring{\tau}^{(1)}_A:=\mathring{\tau}_A\text{ and }\mathring{\tau}_A^{(i)}:=\mathring{\tau}_A\circ \theta_{\mathring{\tau}^{(i-1)}_A}+\mathring{\tau}^{(i-1)}_A,\text{ for }i\geq 2\text{ iteratively.}\] 

For simplicity, we write $\tau_a=\tau_{\{a\}},\mathring{\tau}_a=\mathring{\tau}_{\{a\}}$ and $\mathring{\tau}^{(i)}_a=\mathring{\tau}^{(i)}_{\{a\}}$. We also write $i\wedge j:=\min\{i,j\}$ for any real number $i,j$. In particular, $\tau_A\wedge\tau_{A'}=\tau_{A\cup A'}$.\vspace{0.15cm}

\noindent\textbf{Remark.} We say that the Markov chain $(\Omega,X_n,\mbP_x)$ is \textbf{irreducible} if for any $x,y\in V$ there is $n\geq 0$ such that $P_n(x,y)>0$. In this case, $\mathbb{P}_x(\mathring{\tau}_A<\infty)=1$ for any $\emptyset\neq A\subset V$ and $x\in V$.\vspace{0.15cm}

Next, we introduce some terminologies that will help simplify the notations in the forthcoming proofs in Subsection \ref{sec2.2}. 
\begin{definition}\label{def23}
	Let	$\bm{w}=(w_0,w_1,\cdots,w_\eta)$ be a finite path on $V$. 
	
	(a). We say $\bw$ is self-avoiding if the points $w_0,w_1,\cdots,w_\eta$ are distinct.
	
	(b). Let $I=(n_0,n_1,\cdots,n_\iota)\subset \{1,2,\cdots,\eta\}$, where we order $0\leq n_0<n_1<\cdots<n_\iota\leq \eta$. We define $\bw|_I=(w_{n_0},w_{n_1},\cdots,w_{n_\iota})$, which is also a finite path on $V$.
	
	(c). We define the reversal $\mfR(\bw)$ of $\bw$ as $\mfR(\bw)=(w_\eta,w_{\eta-1},\cdots,w_1,w_0)$. 
	
	(d). Let $\bm{w}'=(w'_0,w'_1,\cdots,w'_{\eta'})$ be a finite path such that $w'_0=w_\eta$, we define the concatenation of $\bm{w}$ and $\bm{w}'$ as 
	\[\bm{w}\oplus \bm{w}'=(w_1,w_2,\cdots,w_{\eta-1},w_{\eta},w'_1,\cdots,w'_{\eta'}).\]
\end{definition}

In particular, with a slight abuse of notation, when we talk about a finite path, we write $[s,t]=\{s,s+1,\cdots,t\}$ for short. We also view the Markov chain $X$ as an infinite path, so the notation $X|_{[s,t]}:=(X_s,X_{s+1},\cdots,X_t)$ (or $X|_{[s,t]}(\omega):=\big(X_s(\omega),X_{s+1}(\omega),\cdots,X_t(\omega)\big)$ if we want to emphasis that we fix an $\omega\in \Omega$) is consistent with  Definition \ref{def23} (b). 

Finally, as a brief review of the LE and PLE operators, we look at one simple example. 

\begin{example}\label{example22}
	Let $V=\{a,b,c,d,e\}$, $\widetilde{V}=\{a,c,d,e\}$ and $\bw=(a,b,c,d,b,e,d)$. 
	
	(a). One can see the LE of $\bw$ is 
	\[\mfL(\bw)=(a,b,e,d),\quad \mfI(\bw)=(1,2,6,7).\]
	
	(b). The PLE associated with $\widetilde{V}$ of $\bw$ is 
	\[\mfL_{\widetilde{V}}(\bw)=(a,b,c,d),\quad \mfI_{\widetilde{V}}(\bw)=(1,2,3,4).\]
	As a consequence, $\mfL\circ \mfL_{\widetilde{V}}(\bw)=(a,b,c,d)$.
	
	In particular, $\mfL(\bw)\neq \mfL\circ \mfL_{\widetilde{V}}(\bw)$ in this example.  
\end{example}

\subsection{Proof of Theorem \ref{th1}}\label{sec2.2} In this subsection, we prove Theorem \ref{th1}.\vspace{0.1cm}

As a first step, we prove a simple case of Theorem \ref{th1}, when $m=2$ and $V_1=V\setminus \{b\}$ for some $b\in V$. For the proof, we introduce two different ways of erasing loops, which we call \textbf{Algorithm 1} and \textbf{Algorithm 2}.\vspace{0.15cm}

\noindent\textbf{Alogrithm 1.} We define $L^{(1)}$ depending on two different cases.

\noindent\emph{Case 1. $\mfL(X|_{[0,\tau_A]})$ doesn't hit $b$.}  In this case, we simply let $L^{(1)}=\mfL(X|_{[0,\tau_A]})$.

\noindent\emph{Case 2. $\mfL(X|_{[0,\tau_A]})$ hits $b$.} In this case, we can find a unique index $n_j$, from the index set  $\mfI(X|_{[0,\tau_A]})=(n_0,n_1,\cdots,n_\iota)$, such that $X_{n_j}=b$. We define 
\[L^{(1)}=(X_{n_0},X_{n_1},\cdots,X_{n_j})\oplus\big(\mfR\circ\mfL\circ\mfR(X|_{[n_j,\tau_A]})\big).\]
In other words, before $n_j$ we erase loops in chronological order, and after $n_j$ we erase loops in reverse order (\textbf{Warning: $n_j\neq\tau_b$ in general}).\vspace{0.15cm}

\noindent\textbf{Alogrithm 2.} We first define $\widetilde{L}^{(2)}$ depending on two different cases. Then, let $L^{(2)}=\mfL(\widetilde{L}^{(2)})$. 

\noindent\emph{Case 1. $\mfL(X|_{[0,\tau_A]})$ doesn't hit $b$.}  In this case, we simply $\widetilde{L}^{(2)}=\mfL_{V_1}(X|_{[0,\tau_A]})$.

\noindent\emph{Case 2. $\mfL(X|_{[0,\tau_A]})$ hits $b$.}  In this case, we can find a unique index $n_j$, from the index set  $\mfI(X|_{[0,\tau_A]})=(n_0,n_1,\cdots,n_\iota)$, such that $X_{n_j}=b$.  We define 
\[\widetilde{L}^{(2)}=(X_{n_0},X_{n_1},\cdots,X_{n_j})\oplus \big(\mfR\circ\mfL_{V_1}\circ\mfR(X|_{[n_j,\tau_A]})\big).\]
In other words, before $n_j$ we erase loops in chronological order, and after $n_j$ we erase loops partially (associated with $V_1$) in reverse order.\vspace{0.15cm}

Although \textbf{Algorithm 2} looks more complicated than \textbf{Algorithm 1}, we claim that the two algorithms are the same (thus have the same distribution).

\begin{lemma}\label{lemma24}
	$L^{(1)}=L^{(2)}$. 
\end{lemma} 
\begin{proof}
If $\mfL(X|_{[0,\tau_A]})$ doesn't hit $b$, one can check that $\widetilde{L}^{(2)}=\mfL(X|_{[0,\tau_A]})$. Hence, $L^{(2)}=\mfL(\widetilde{L}^{(2)})=\mfL(X|_{[0,\tau_A]})=L^{(1)}$ in this case. It remains to consider the case that $\mfL(X|_{[0,\tau_A]})$ hits $b$. In this case, $\mfL\circ\mfR(X|_{[n_j,\tau_A]})$ and $\mfL_{V_1}\circ\mfR(X|_{[n_j,\tau_A]})$ take the form 
\[\begin{cases}
\mfL\circ\mfR(X|_{[n_j,\tau_A]})=(w_0,w_1,\cdots,w_{j'},b),\\
\mfL_{V_1}\circ\mfR(X|_{[n_j,\tau_A]})=(w_0,w_1,\cdots,w_{j'},b,v_0,\cdots,v_{j''}),
\end{cases}
\] 
for some $j',j''\geq -1$, some $w_0,\cdots,w_{j'}\in V\setminus \{b, X_{n_0},\cdots,X_{n_{j-1}}\}$ and some $v_0,\cdots,v_{j''}\in V\setminus \{X_{n_0},\cdots,X_{n_{j-1}}\}$ with $v_{j''}=b$ (if $j''\geq 0$). Hence 
\[\begin{aligned}
L^{(2)}&=\mfL(\widetilde{L}^{(2)})=\mfL\big((X_{n_0},X_{n_1},\cdots,X_{n_j})\oplus \big(\mfR\circ\mfL_{V_1}\circ\mfR(X|_{[n_j,\tau_A]})\big)\big)\\
&=\mfL(X_{n_0},X_{n_1},\cdots,X_{n_{j-1}},X_{n_j}=b=v_{j''},\cdots,v_0,b,w_{j'},\cdots,w_1,w_0)\\
&=(X_{n_0},X_{n_1},\cdots,X_{n_{j-1}},b,w_{j'},\cdots,w_1,w_0)\\
&=(X_{n_0},X_{n_1},\cdots,X_{n_{j-1}},X_{n_j})\oplus \big(\mfR\circ\mfL\circ\mfR(X|_{[n_j,\tau_A]})\big)=L^{(1)}.
\end{aligned}\]
\end{proof}

It remains to show that $L^{(1)}$ and $\mfL(X|_{[0,\tau_A]})$ have the same law, and show that $L^{(2)}$ and $\mfL\circ\mfL_{V_1}(X|_{[0,\tau_A]})$ have the same law. We will use an idea of Lawler \cite{L5} to compute the probability of the paths with Green's function. In particular, Lawler successfully showed that $\mfR\circ \mfL\circ \mfR(X|_{[0,\tau_A]})$ and $\mfL(X|_{[0,\tau_A]})$ have the same law. \vspace{0.1cm}
 
\noindent\textbf{Green's function.} Let $B\subsetneq V$, we define the Green's function on $B^2=B\times B$ as 
\[G_B(x,y)=\mathbb{E}_x(\#\{0\leq n<\tau_{V\setminus B}:X_n=y\}),\qquad \forall x,y\in  B.\]
By the strong Markov property, $G_B(x,y)=\mathbb{P}_{x}(\tau_y<\tau_{V\setminus B})G_B(y,y)$.\vspace{0.1cm}

\begin{lemma}[\cite{L5}]\label{lemma25}
Let $B\subsetneq V$, $x,y\in B$ such that $x\neq y$, then
  $G_{B\setminus\{y\}}(x,x)\cdot G_B(y,y)=G_B(x,x)\cdot G_{B\setminus\{x\}}(y,y)$. 
\end{lemma}
\begin{proof}
For the convenience of readers, we provide a proof here. By the strong Markov property,
\[
\begin{aligned}
&G_B(x,x)=\mathbb{E}_x(\#\{0\leq n<\tau_{V\setminus B}:X_n=x\})\\
=&\mathbb{E}_x(\#\{0\leq n<\tau_{V\setminus B}\wedge \tau_y:X_n=x\})+\mathbb{E}_x(\#\{\tau_y\leq n<\tau_{V\setminus B}:X_n=x\})\\
=&G_{B\setminus\{y\}}(x,x)+\mathbb{P}_{x}(\tau_y<\tau_{V\setminus B})G_B(y,x),
\end{aligned}
\]
Hence, $G_{B\setminus\{y\}}(x,x)=\big(1-\mbP_x(\tau_y<\tau_{V\setminus B})\mbP_y(\tau_x<\tau_{V\setminus B})\big)G_B(x,x)$. We also have $G_{B\setminus\{x\}}(y,y)=\big(1-\mbP_y(\tau_x<\tau_{V\setminus B})\mbP_x(\tau_y<\tau_{V\setminus B})\big)G_B(y,y)$ by a same argument. The lemma follows immediately.
\end{proof}

The observation of Lawler \cite{L5}, by using Lemma \ref{lemma25} multiple times, is \[F_B(x_0,x_1,\cdots,x_n)=F_B(x_{\sigma(0)},x_{\sigma(1)},\cdots,x_{\sigma(n)}),\]
for any $n\geq 1$, any $\{y_0,y_1,\cdots,y_n\}\subset B\subsetneq V$, and any permutation $\sigma$ of $\{0,1,\cdots,n\}$, where \[F_B(y_0,y_1,\cdots,y_n)=G_B(y_0,y_0)\cdot G_{B\setminus \{y_0\}}(y_1,y_1)\cdots G_{B\setminus \{y_0,\cdots,y_{n-1}\}}(y_n,y_n).\]

\begin{proposition}\label{prop26}
	Let $\big(\Omega,X_n,\mbP_x\big)$ be a Markov chain on a finite set $V$. Let $b\in V$, $V_1=V\setminus \{b\}$ and $V_2=V$. Then for any $A\subset V$ and $x\in V$ such that $\mathbb{P}_x(\tau_A<\infty)=1$, we have 
	\[\mbP_x\big(\mfL(X|_{[0,\tau_A]})\in \bullet\big)=\mbP_x\big(\mfL_{V_2}\circ \mfL_{V_1}(X|_{[0,\tau_A]})\in \bullet\big).\]
\end{proposition}
\begin{proof}
We only need to consider the case that $b,x\notin A$. For short, we write $A	^c=V\setminus A$. 
	
First, we show $\mathbb{P}_x(L^{(1)}=\bw)=\mathbb{P}_x\big(\mfL(X|_{[0,\tau_A]})=\bw\big)$, where we fix a finite self-avoiding path $\bw=(w_0,w_1,\cdots,w_\eta)$ with $w_0=x$, $w_\eta\in A$ and $w_n\in A^c,\forall 0\leq n<\eta$. Noticing that $\mfL(X|_{[0,\tau_A]})$ hits $b$ if and only if $L^{(1)}$ hits $b$, and $\mfL(X|_{[0,\tau_A]})=L^{(1)}$ in this case, we only need to consider those $\bw$ such that $w_j=b$ for some $0\leq j<\eta$. Following  Lawler \cite{L5}, we use the Green's function to compute the probability, and for the convenience of readers, we show some details when we do this for the first time: 
\[\begin{aligned}
&\mbP_x\big(\mfL(X|_{[0,\tau_A]})=\bw\big)\\
=&\sum_{0<l_1<\cdots<l_\eta}\mbP_x\big(\mfL(X|_{[0,\tau_A]})=\bw,\ \mfI(X|_{[0,\tau_A]})=(0,l_1,\cdots,l_\eta)\big)\\
=&\sum_{0<l_1<\cdots<l_\eta}\mbP_{w_0}\big(X_{l_1-1}=w_0,\tau_A>l_1-1\big)\cdot P(w_0,w_1)\\
&\qquad\qquad\quad\cdot\mbP_{w_1}\big(X_{l_2-l_1-1}=w_1,\tau_{A\cup \{w_0\}}>l_2-l_1-1\big)\cdots P(w_{\eta-1},w_\eta)\\
=&\prod_{n=0}^{\eta-1}\big(G_{A^c\setminus \{w_0,\cdots,w_{n-1}\}}(w_n,w_n)\cdot P(w_n,w_{n+1})\big).
\end{aligned}\] 
By a similar argument, one can also see
\[
\begin{aligned}
\mbP_x(L^{(1)}=\bw)=&\big(\prod_{n=0}^{\eta-1}P(w_n,w_{n+1})\big)\cdot\big(\prod_{n=0}^{j-1}G_{A^c\setminus \{w_0,\cdots,w_{n-1}\}}(w_n,w_n)\big)\\
&\cdot \big(\prod_{n=1}^{\eta-j}G_{A^c\setminus  \{w_0,\cdots,w_{j-1},w_{\eta-1},\cdots,w_{\eta-n+1}\}}(w_{\eta-n},w_{\eta-n})\big),
\end{aligned}
\]
where $w_{\eta-1},\cdots w_{\eta-n+1}$ is the empty sequence if $n=1$. Hence, $\mathbb{P}_x(L^{(1)}=\bw)=\mathbb{P}_x\big(\mfL(X|_{[0,\tau_A]})=\bw\big)$ by applying Lemma \ref{lemma25}.

Next, we show $\mathbb{P}_x(L^{(2)}\in\bullet)=\mathbb{P}_x\big(\mfL\circ\mfL_{V_1}(X|_{[0,\tau_A]})\in\bullet\big)$. It suffices to show  $\mathbb{P}_x(\widetilde{L}^{(2)}\in\bullet)=\mathbb{P}_x\big(\mfL_{V_1}(X|_{[0,\tau_A]})\in\bullet\big)$. Let $\bw=(w_0,w_1,\cdots,w_\eta)$ be a finite path with $w_0=x$, $w_\eta\in A$  and $w_n\in A^c,\forall 0\leq n<\eta$, and in addition $\mfL_{V_1}(\bw)=\bw$. Still, we are only interested in the case that $\bw$ hits $b$, and we let $(w_{n_0},w_{n_1},\cdots,w_{n_\iota})$ be the subpath of $\bw$ that consists of all the elements that $\bw$ hits in $V_1$ in chronological order. Similar calculations as before imply
\[\begin{aligned}
	\mbP_x\big(\mfL_{V_1}(X|_{[0,\tau_A]})=\bw\big)=\big(\prod_{n=0}^{\eta-1}P(w_i,w_{i+1})\big)\big(\prod_{i=0}^{\iota-1}G_{A^c\setminus  \{w_{n_0},\cdots,w_{n_{i-1}}\}}(w_{n_i},w_{n_i})\big) ,
\end{aligned}\]
To compute the probability of $L^{(2)}=\bw$, we need a little more information about $\bw$: we assume $w_0=0,\cdots,w_{n_{j-1}}=j-1$ and $w_j=b$, which means $w_j$ is the first term of $\bw$ that takes the value $b$. Then by some careful computations as before, one can see
\[
\begin{aligned}
	\mbP_x(\widetilde{L}^{(2)}=\bw)=&\big(\prod_{n=0}^{\eta-1}P(w_n,w_{n+1})\big)\cdot\big(\prod_{i=0}^{j-1}G_{A^c\setminus \{w_{n_0},\cdots,w_{n_{i-1}}\}}(w_{n_i},w_{n_i})\big)\\
	&\cdot \big(\prod_{i=1}^{\iota-j}G_{A^c\setminus  \{w_{n_0},\cdots,w_{n_{j-1}},w_{n_{\iota-1}},\cdots,w_{n_{\iota-i+1}}\}}(w_{n_{\iota-i}},w_{n_{\iota-i}})\big).
\end{aligned}
\]
Again by Lemma \ref{lemma25}, we can see $\mathbb{P}_x(\widetilde{L}^{(2)}=\bw)=\mathbb{P}_x\big(\mfL_{V_1}(X|_{[0,\tau_A]})=\bw\big)$.

Finally, one applies Lemma \ref{lemma24} to see that $\mfL(X|_{[0,\tau_A]})$, $L^{(1)}=L^{(2)}$ and $\mfL\circ\mfL_{V_1}(X|_{[0,\tau_A]})$ have the same distribution. 
\end{proof}

The next step is to iterate the simple case. The key observations are Lemma\ref{lemma27} and \ref{lemma28}, as easy consequences of the strong Markov property. 

\begin{lemma}\label{lemma27}
Let $(\Omega,X_n,P_x)$ be a Markov chain on a finite set $V$, and let $A\subset V$. Define $V_A=\{x\in V:\mathbb{P}_x(\tau_A<\infty)=1\}$ and $\Omega_A=\{\omega\in \Omega:X_n(\omega)\in V_A,\ \forall 0\leq n\leq \tau_A(\omega)\}$. Then, for each $\widetilde{V}\subset V$, $(\Omega_A,\widetilde{X}_n,\mathbb{P}_x)$ is a Markov chain on $V_A$, where $\widetilde{X}_0=X_0$ and $\widetilde{X}_n=X_{\mathring{\tau}_{\widetilde{V}}^{(n)}\wedge \tau_A}$ for $n\geq 1$. 
\end{lemma}
\noindent\textbf{Remark.} Not all the points of $\widetilde{V}$ but only those in $V_A\cap \widetilde{V}$ are involved in the new process.
\begin{proof}
Clearly, $\mathbb{P}_x(\Omega_A)=1$ for each $x\in V_A$. The lemma then follows from the strong Markov property, and the fact $t_n=t_1\circ\theta_{t_{n-1}}+t_{n-1}$, where $t_0=0$ and $t_n=\mathring{\tau}_{\widetilde{V}}^{(n)}\wedge \tau_A$ for $n\geq 1$. In fact, the transition kernel associated with $(\Omega,\widetilde{X}_n,P_x)$ is $\widetilde{P}(x,y)=\mathbb{P}_x(X_{\mathring{\tau}_{\widetilde{V}}\wedge \tau_A}=y)$ for any $x,y\in V_A$.
\end{proof}

In the statement of the following Lemma \ref{lemma28}, the definition of $\tau_A$ depends on the process that we are talking about. 
\begin{lemma}\label{lemma28}
We assume the same settings as in Lemma \ref{lemma27}, and let $V_1\subset V_2=\widetilde{V}$. 

(a). For $\omega\in\Omega_A$ such that $X_0(\omega)\notin A$, we can decompose $\mfL_{V_2}\circ\mfL_{V_1}(X|_{[0,\tau_A]})$ conditioned on $\mfL_{V_2}\circ\mfL_{V_1}(\widetilde{X}|_{[0,\tau_A]})$ as follows,
\[\begin{cases}
	\mfL_{V_2}\circ\mfL_{V_1}(\widetilde{X}|_{[0,\tau_A]})=(y_0,y_1,\cdots,y_\iota),\\
	\mfL_{V_2}\circ\mfL_{V_1}(X|_{[0,\tau_A]})=\bm{v}_1\oplus\bm{v}_2\oplus\cdots\oplus\bm{v}_\iota,
\end{cases}
\]
where for $1\leq i\leq \iota$, $\bm{v}_i$ takes the form $\bm{v}_i=(y_{i-1},\cdots,y_i)$, with all points in `$\cdots$' lying in $V\setminus (\widetilde{V}\cup A)$. 

(b). $\bm{v}_i,1\leq i\leq \iota$ are mutually independent conditioned on $\mfL_{V_2}\circ\mfL_{V_1}(\widetilde{X}|_{[0,\tau_A]})$. In addition, for each $1\leq i\leq \iota$, 
\begin{align*}
&\mathbb{P}_x\big(\bm{v}_i\in \bullet\big|\mfL_{V_2}\circ\mfL_{V_1}(\widetilde{X}|_{[0,\tau_A]})=(y_0,\cdots,y_{i-1},y_i,\cdots,y_\iota),\iota\geq i\big)\\
=&\mathbb{P}_{y_{i-1}}\big(X|_{[0,\mathring{\tau}_{\widetilde{V}}\wedge \tau_A]}\in\bullet\big|\widetilde{X}_1=y_i\big). 
\end{align*}
\end{lemma}

\begin{proof}
(a) is obvious. To show (b), we consider the distribution of $\bm{v}_i$ conditioned on $\widetilde{X}|_{[0,\tau_A]}$. We let $\bm{y}=(y_0,y_1,\cdots,y_\iota)$ and $\bm{w}=(w_0,w_1,\cdots,w_\eta)$ such that $\mfL_{V_2}\circ\mfL_{V_1}(\bm{w})=\bm{y}$, and we write 
\[(n_0,n_1,\cdots,n_\iota)=(k_{l_0},k_{l_1},\cdots,k_{l_\iota}),\]
where $\mfI_{V_1}(\bw)=(k_0,k_1,\cdots,k_{\iota'}),\ \mfI_{V_2}\circ\mfL_{V_1}(\bw)=(l_0,l_1,\cdots,l_{\iota})$.

Then, if $\widetilde{X}|_{[0,\tau_A]}=\bm{w}$, for $1\leq i\leq \iota$, we have 
\[\bm{v}_i=
\begin{cases}
	(X_0,X_1,\cdots,X_{\mathring{\tau}_{\widetilde{V}}\wedge \tau_A}),&\text{ if }i=1\text{ and }n_1=1,\\
	(X_{\mathring{\tau}^{(n_i-1)}_{\widetilde{V}}\wedge \tau_A},\cdots,X_{\mathring{\tau}^{(n_i)}_{\widetilde{V}}\wedge \tau_A}),&\text{ otherwise.}
\end{cases}
\]
Whence, by the strong Markov property, $\bm{v}_i,1\leq i\leq \iota$ are mutually independent conditioned on the event $\widetilde{X}|_{[0,\tau_A]}=\bw$, and 
\[
 \mathbb{P}_x\big(\bm{v}_i\in \bullet\big|\widetilde{X}|_{[0,\tau_A]}=\bw\big)
 =\mathbb{P}_{y_{i-1}}\big(X|_{[0,\mathring{\tau}_{\widetilde{V}}\wedge \tau_A]}\in\bullet\big|\widetilde{X}_1=y_i\big),\qquad\forall 1\leq i\leq \iota. 
\]
(b) then follows easily noticing that the conditional law depends only on $\mfL_{V_2}\circ\mfL_{V_1}(\widetilde{X}|_{[0,\tau_A]})$.
\end{proof}

In particular, Lemma \ref{lemma28} implies that the law of $\mfL_{V_2}\circ \mfL_{V_1}(X|_{[0,\tau_A]})$ depends only on the law of $\mfL_{V_2}\circ \mfL_{V_1}(\widetilde{X}|_{[0,\tau_A]})$. Combining with Lemma \ref{lemma27}, we are now able to iterate the special case proved in Proposition \ref{prop26}. 

\begin{lemma}\label{lemma29}
Let $(\Omega,X_n,\mbP_x)$ be a discrete time Markov chain on a finite set $V$.  Let 
\[\emptyset\neq V_1\subsetneq V_2\subsetneq \cdots\subsetneq V_m=V\]
be an expanding sequence such that $\#V_j\setminus V_{j-1}=1,\forall 2\leq j\leq m$. Then for any $A\subset V$ and $x\in V$ such that $\mbP_x(\tau_A<\infty)=1$, we have
\[\mbP_x\big(\mfL(X|_{[0,\tau_A]})\in \bullet\big)=\mbP_x\big(\mfL_{V_m}\circ\cdots\circ\mfL_{V_2}\circ \mfL_{V_1}(X|_{[0,\tau_A]})\in \bullet\big).\]
\end{lemma}
\begin{proof}
Without loss of generality, we assume $A\subset V_j,\forall j=2,3,\cdots,m$.

First, if we let $\widetilde{V}=V_2$ in Lemma \ref{lemma27}. By Proposition \ref{prop26}, we can show that $\mfL_{V_2}\circ\mfL_{V_1}(\widetilde{X}|_{[0,\tau_A]})$ and $\mfL_{V_2}(\widetilde{X}|_{[0,\tau_A]})$ have the same law. Then, if $x\in V_2$, $(\Omega_A,\widetilde{X}_n,\mathbb{P}_y)$ can be viewed as a Markov chain on $V_2\cap V_A$ (restricted on a smaller sample space if necessary), so the claim follows from Proposition \ref{prop26} and Lemma \ref{lemma27}; if $x\notin V_2$, $(\Omega_A,\widetilde{X}_n,\mathbb{P}_y)$ can be viewed as a Markov chain on $(V_2\cup\{x\})\cap V_A$ (restricted on a smaller sample space if necessary), and we get the same conclusion noticing that $\mathbb{P}_x\big(\mfL_{V_2}(\widetilde{X}|_{[0,\tau_A]})\in \bullet\big)=\mathbb{P}_x\big(\mfL_{V_2\cup\{x\}}(\widetilde{X}|_{[0,\tau_A]})\in \bullet\big)$ and  $\mathbb{P}_x\big(\mfL_{V_2}\circ \mfL_{V_1}(\widetilde{X}|_{[0,\tau_A]})\in \bullet\big)=\mathbb{P}_x\big(\mfL_{V_2\cup\{x\}}\circ \mfL_{V_1\cup\{x\}}(\widetilde{X}|_{[0,\tau_A]})\in \bullet\big)$. 

Next, by Lemma \ref{lemma28}, $\mbP_x\big(\mfL_{V_2}\circ \mfL_{V_1}(X|_{[0,\tau_A]})\in \bullet\big)=\mbP_x\big(\mfL_{V_2}(X|_{[0,\tau_A]})\in \bullet\big)$.
As a consequence,
\[
\begin{aligned}
&\mbP_x\big(\mfL_{V_m}\circ\cdots\circ\mfL_{V_3}\circ\mfL_{V_2}\circ \mfL_{V_1}(X|_{[0,\tau_A]})\in \bullet\big)\\
=&\mbP_x\big(\mfL_{V_m}\circ\cdots\circ\mfL_{V_3}\circ \mfL_{V_2}(X|_{[0,\tau_A]})\in \bullet\big)\\
=&\cdots=\mbP_x\big(\mfL_{V_m}(X|_{[0,\tau_A]})\in \bullet\big),
\end{aligned}
\]
where we repeat the same argument in `$\cdots$'.
\end{proof}

The difference between Lemma \ref{lemma29} and Theorem \ref{th1} is that we have the additional assumption $\#V_j\setminus V_{j-1}=1$. The remaining difficulty is to shortening the sequence. 

\begin{proposition}\label{prop210}
Let $(\Omega,X_n,\mbP_x\big)$ be a Markov chain on a finite set $V$. Let $\emptyset\neq V_1\subsetneq V_2=V$. Then for any $A\subset V$ and $x\in V$ such that $\mathbb{P}_x(\tau_A<\infty)=1$, we have 
\[\mbP_x\big(\mfL(X|_{[0,\tau_A]})\in \bullet\big)=\mbP_x\big(\mfL_{V_2}\circ \mfL_{V_1}(X|_{[0,\tau_A]})\in \bullet\big).\]
\end{proposition}
\begin{proof}
For convenience, we let $m=\#(V\setminus V_1)+1$. For each self-avoiding path $\bm{v}=(v_0,v_1,\cdots,v_k)$ on $V\setminus V_1$, we fix a sequence
\[V_1=V_{1,\bm{v}}\subsetneq V_{2,\bm{v}}\subsetneq \cdots\subsetneq V_{m,\bm{v}}=V\]
such that $V_{j+2,\bm{v}}\setminus V_{j+1,\bm{v}}=\{v_j\}$ for any $0\leq j\leq k$. 

We use $\Xi$ to denote the set of finite paths on $V$ generated by $\mfL_{V_1}$, i.e. $\Xi=\{\bw:\bw=\mfL_{V_1}(\bw'),\ \bw'\text{ is a finite path on }V\}$. For each self-avoiding path $\bw$, we write $\bw^{[V\setminus V_1]}$ for the subpath of $\bw$ consisted of all the elements that $\bw$ hits in $V\setminus V_1$ in chronological order. For each self-avoiding path $\bm{v}$ on $V\setminus V_1$, we define
\[\Xi_{\bm{v}}=\{\bw\in \Xi: (\mfL\bw)^{[V\setminus V_1]}=\bm{v}\},\]
and 
\[\Omega_{\bm{v}}=\{\omega\in \Omega: \mfL_{V_1}\big(X(\omega)|_{[0,\tau_A]}\big)\in \Xi_{\bm{v}}\big\}.\]
We have the following two observations. \vspace{0.15cm}

\noindent\textit{Observation 1. For each self-avoiding path $\bm{v}$ on $V\setminus V_1$ and $\bw\in \Xi_{\bm{v}}$, we have}
\[\mfL(\bw)=\mfL_{V_{m,\bm{v}}}\circ\cdots\circ\mfL_{V_{3,\bm{v}}}\circ\mfL_{V_{2,\bm{v}}}(\bw).\]

\noindent\textit{Observation 2. For each $\bw\in \Xi\setminus \Xi_{\bm{v}}$, we have}
\[\big(\mfL_{V_{m,\bm{v}}}\circ\cdots\circ\mfL_{V_{3,\bm{v}}}\circ\mfL_{V_{2,\bm{v}}}(\bw)\big)^{[V\setminus V_1]}\neq \bm{v}.\] 

Observation 1 is clear. We explain Observation 2 here. For $\bw\in \Xi\setminus \Xi_{\bm{v}}$, we write $\bm{v}'=(v'_0,\cdots,v'_{k})=\mfL(\bw)^{[V\setminus V_1]}$. There are two possible cases. In the first case, when there is $0\leq j\leq k$ such that $\{v'_j\}\neq V_{j+2,\bm{v}}\setminus V_{j+1,\bm{v}}$ and $\{v'_i\}= V_{i+2,\bm{v}}\setminus V_{i+1,\bm{v}},\forall 0\leq i<j$, one can easily check that $\big(\mfL_{V_{m,\bm{v}}}\circ\cdots\circ\mfL_{V_{3,\bm{v}}}\circ\mfL_{V_{2,\bm{v}}}(\bw)\big)^{[V\setminus V_1]}$ takes the form $(v'_0,v'_1,\cdots,v'_j,\cdots)$, hence the observation holds. In the second case, if no such $j$ appears, one simply have $\mfL_{V_{m,\bm{v}}}\circ\cdots\circ\mfL_{V_{3,\bm{v}}}\circ\mfL_{V_{2,\bm{v}}}(\bw)=\mfL(\bw)$, so the observation also holds. \vspace{0.15cm}

Finally, the theorem follows from the following sequence of equalities, which holds for any self-avoiding path $\bw$ and $x\in V$,
\begin{align*}
&\mathbb{P}_x\big(\mfL\circ\mfL_{V_1}(X|_{[0,\tau_A]})=\bw\big)\\
=&\mathbb{P}_x\big(\Omega_{\bm{v}},\ \mfL\circ\mfL_{V_1}(X|_{[0,\tau_A]})=\bw\big)\\
=&\mathbb{P}_x\big(\Omega_{\bm{v}},\ \mfL_{V_{m,\bm{v}}}\circ\cdots\circ\mfL_{V_{2,\bm{v}}}\circ\mfL_{V_{1,\bm{v}}}(X|_{[0,\tau_A]})=\bm{w}\big)\\
=&\mathbb{P}_x\big(\mfL_{V_{m,\bm{v}}}\circ\cdots\circ\mfL_{V_{2,\bm{v}}}\circ\mfL_{V_{1,\bm{v}}}(X|_{[0,\tau_A]})=\bm{w}\big)\\
=&\mathbb{P}_x\big(\mfL(X|_{[0,\tau_A]})=\bm{w}\big),
\end{align*}
where $\bm{v}=\bw^{[V\setminus V_1]}$. We briefly explain the equality for the convenience of readers: the first equality holds by the definition of $\Omega_{\bm{v}}$; the second equality holds due to Observation 1; the third equality holds due to Observation 2; the last equality holds due to Lemma \ref{lemma29}.
\end{proof}

\begin{proof}[Proof of Theorem \ref{th1}]
	Theorem \ref{th1} follows from Proposition \ref{prop210} by a same argument as the proof of Lemma \ref{lemma29}.
\end{proof}

\subsection{The Sierpi\'nski gasket graphs}
Finally, as an example, we briefly explain how to apply Theorem \ref{th1} to Sierpi\'nski gasket graphs. The algorithm  `erasing-larger-loops-first' introduced in \cite{H,HM} is more complicated to describe, so we will not explain it here. Readers can also find related study of self-avoiding walks on Sierpi\'nski gaskets in \cite{HHH,HH,HHK}. 

Let $q_1=(0,0)$, $q_2=(1,0)$ and $q_3=(\frac{1}{2},\frac{\sqrt{3}}{2})$ be the three vertices of an equilateral triangle in $\mathbb{R}^2$, and define the iterated function system (i.f.s.) $\{F_i\}_{i=1}^3$ as $F_i(x)=\frac{1}{2}x+\frac{1}{2}q_i$ for $i=1,2,3$. Then the Sierpi\'nski gasket $\mathcal{SG}$ is the unique compact subset of $\mathbb{R}^2$ such that 
$\mathcal{SG}=F_1(\mathcal{SG})\cup F_2(\mathcal{SG})\cup F_3(\mathcal{SG})$. See Figure \ref{SG} for a picture of $\mathcal{SG}$. 

\begin{figure}[htp]
	\includegraphics[width=3.5cm]{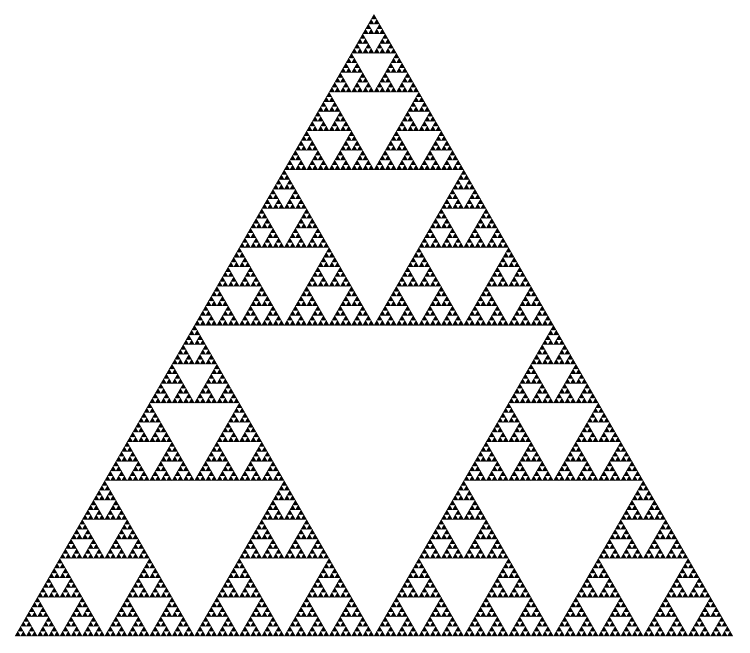}
	\caption{The Sierpi\'nski gasket.}\label{SG}
\end{figure}

$\mathcal{SG}$ can naturally be approximated by a sequence of graphs. Let $V_0=\{q_1,q_2,q_3\}$, and iteratively for $m\geq 1$, we let $V_m=\bigcup_{i=1}^3F_i(V_{m-1})$. The associated (undirected) edge sets are defined as $E_m=\big\{\{p,q\}:\text{there exists }\theta=\theta_1\theta_2\cdots\theta_m\in \{1,2,3\}^m\text{ such that }p,q\in F_{\theta_1}F_{\theta_2}\cdots F_{\theta_m}(V_0)\big\}$. We then call $G_m=(V_m,E_m)$ the level-$m$ Sierpi\'nski gasket graph. See Figure \ref{SGgraphs} for the level-$1,2,3$ Sierpi\'nski gasket graphs.

\begin{figure}[htp]
	\includegraphics[width=3.5cm]{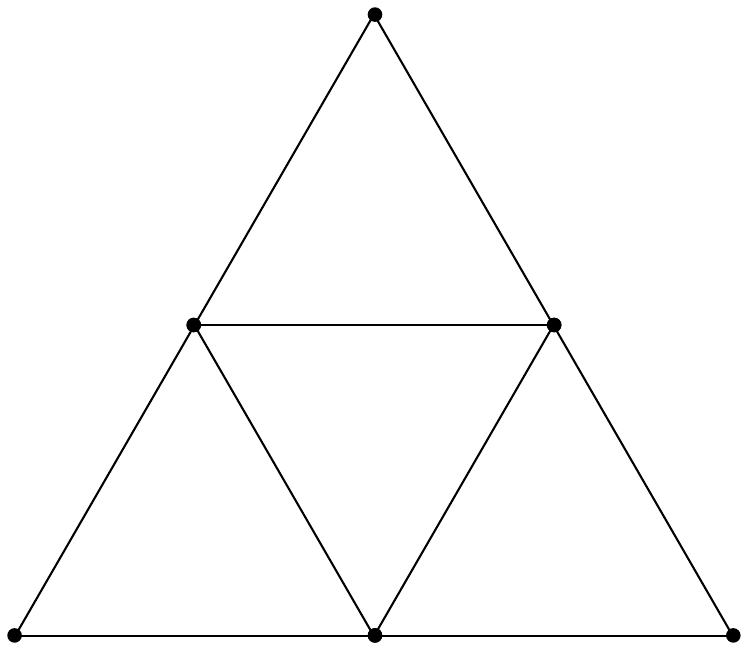}\ 
	\includegraphics[width=3.5cm]{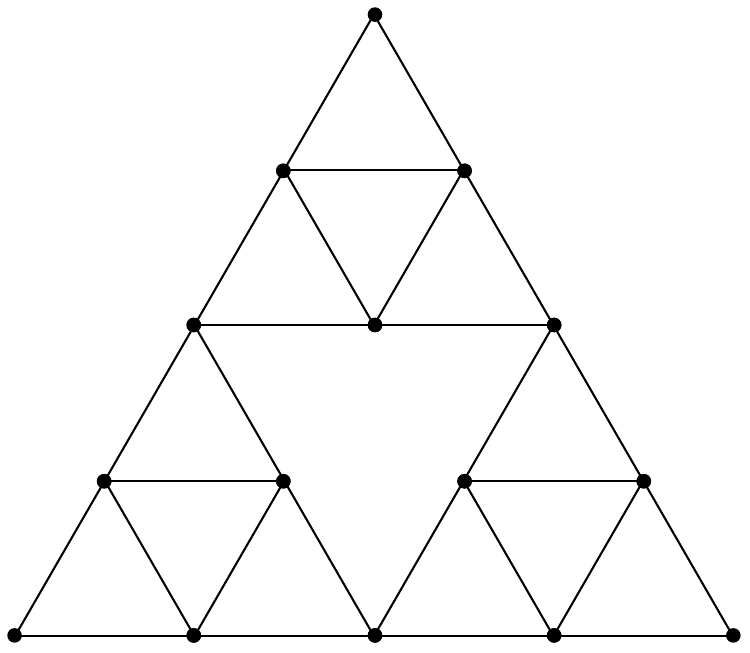}\ 
	\includegraphics[width=3.5cm]{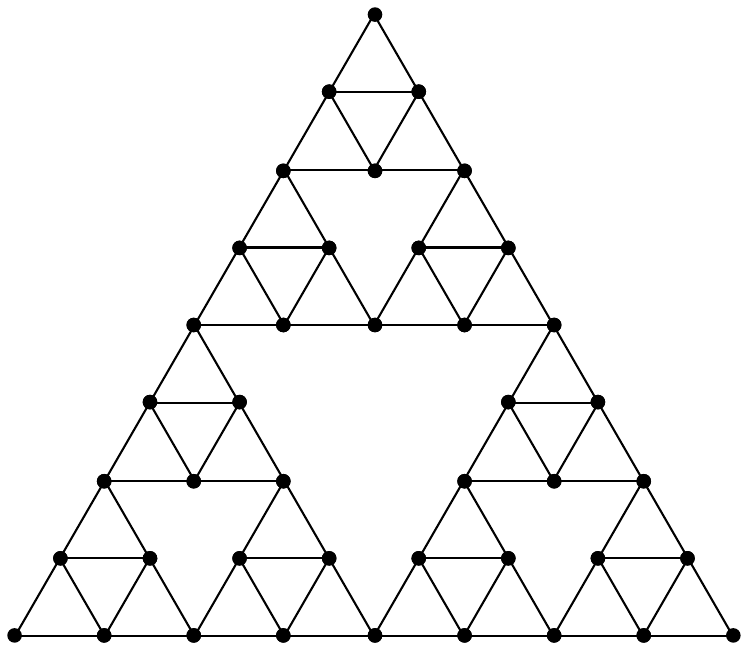}
	\caption{The Sierpi\'nski gasket graphs.}\label{SGgraphs}
\end{figure}

We now consider the simple random walk $(\Omega^{(m)},X^{(m)}_n,\mathbb{P}_x^{(m)})$ on $G_m$, i.e. each step $X_{n+1}^{(m)}$ will be on one of the neighbours of $X^{(m)}_n$ with equal probability. We can apply Theorem \ref{th1} to study the loop-erased walk on $G_m$: 
\[
\mbP_{q_1}\big(\mfL(X^{(m)}|_{[0,\tau_{\{q_2,q_3\}}]})\in \bullet\big)=\mbP_{q_1}\big(\mfL_{V_m}\circ\cdots\circ\mfL_{V_2}\circ \mfL_{V_1}(X^{(m)}|_{[0,\tau_{\{q_2,q_3\}}]})\in \bullet\big),
\]
where $V_1,V_2,\cdots,V_m$ are the vertice sets of Sierpi\'nski gasket graphs. 

Finally, we observe that the trace of $\mfL_{V_{m-1}}\circ\cdots\circ\mfL_{V_2}\circ \mfL_{V_1}(X^{(m)}|_{[0,\tau_{\{q_2,q_3\}}]})$ on $V_{m-1}$ is the LERW on $V_{m-1}$. By the strong Markov property as in Lemma \ref{lemma28}, and thanks to the good geometry of Sierpi\'nski gaskets (in particular, one can easily check that $\mfL_{V_j}\circ\mfL_{V_{j-1}}\circ\cdots\circ \mfL_{V_{1}}(X^{(m)}|_{[0,\tau_{\{q_2,q_3\}}]})$ will hit the same $j-1$ level cells as $\mfL_{V_{j-1}}\circ\cdots\circ \mfL_{V_{1}}(X^{(m)}|_{[0,\tau_{\{q_2,q_3\}}]})$), the problem can be reduced to a study of a multi-branching process. See \cite{H,HM} for detailed calculations, where the exact growth exponent was shown to be $\frac{\log 2}{\log (20+\sqrt{205})-\log15}$ (firsted computed in \cite{STW}).

\section{A review of resistance spaces}
In history, the diffusion process on the Sierpi\'nski gasket was first constructed as the weak limit of reversible random walks on graphs approximating the fractal \cite{BP,G,kus}, and the idea also achieved success on nested fractals \cite{Lindstrom}. Later, J. Kigami introduced the analytical method of constructing Dirichlet form as the limit of energies on approximating graphs, and with the new method, he extended the result to a larger class named post critically finite (p.c.f.) fractals \cite{ki1,ki2}. The structures therein, which were named resistance forms and resistance metrics later \cite{ki3,ki4,ki5}, has now been a fundamental concept in the area of analysis on fractals. 

In this section, we briefly review the definition of resistance spaces. We also introduce some notations and basic tools for the development in Section 4. Readers can also read books \cite{B,ki3,St}  about fractals and resistance forms. 

\subsection{Resistance spaces} The resistance spaces are natural generalizations of the  electrical networks, so it's natural to start our discussions with the discrete cases.\vspace{0.15cm}

\noindent\textbf{\emph{Reversible random walk}}. A Markov kernel $P$ on a countable set $V$ is called reversible if there exists a weight $\{c_x\}_{x\in V}\in (0,\infty)^{V}$ (unique up to a constant multiplier) such that 
\[c_xP(x,y)=c_yP(y,x),\qquad \forall x,y\in V.\] 
The associated Markov chain $\big(\Omega,X_n,\mbP_x\big)$ is called a reversible random walk on $V$. \textbf{In this paper, we also require $P(x,x)=0,\forall x\in V$ all the time.}\vspace{0.15cm}

\noindent\textbf{\emph{(Irreducible) Electrical network}}. Let $P$ be a reversible Markov kernel, and we assume $\big(\Omega,X_n,\mbP_x\big)$ is irreducible. In particular, we need to highlight that irreducibility is a requirement in the definition of the resistance form. We define 
\[c_{x,y}=c_xP(x,y)=c_yP(y,x),\qquad\forall x\neq y.\]
We interpret $c_{x,y}=c_{y,x}$ as the conductance between $x,y$. One can view this structure as an electrical network. 

There is an energy form associated with the electrical network
\[\mcD(f,f)=\frac{1}{2}\sum_{x\in V}\sum_{y\in V\setminus \{x\}}c_{x,y}\big(f(x)-f(y)\big)^2,\quad\forall f\in l(V).\]
Let $\mcF=\{f\in l(V):\mcD(f,f)<\infty\}$. Then, $\mcD$ induces a bilinear form on $\mcF$: $\mcD(f,g)=\frac{1}{4}\big(\mcD(f+g)-\mcD(f-g)\big),\forall f,g\in \mcF$, and we call $(\mcD,\mcF)$ a \emph{(discrete) resistance form}. Moreover, $L^2(V,c):=\{f\in l(V):\sum_{x\in V}f^2(x)c_x<\infty\}\subset\mcF$, and $\big(\mcD, L^2(V,c)\big)$ is a Dirichlet form on $L^2(V,c)$. 

Let $x\neq y\in V$, we define the effective resistance $R$ between $x,y$ as
\[R(x,y)=\big(\inf\{\mcD(f,f):f\in l(V),f(x)=0,f(y)=1\}\big)^{-1}.\]
In addition, we set $R(x,x)=0,\forall x\in V$. \vspace{0.15cm}

Kigami \cite{ki3} showed that the effective resistance $R$ is a metric on $V$. For this reason, we also call $R$ the resistance metric. 

\begin{proposition}[\cite{ki3}]\label{prop31}
	The effective resistance $R$ associated with the resistance form $(\mcD,\mcF)$ is a metric on $V$. In addition, let $(\mcD',\mcF')$ be a different resistance form on $V$, and let $R'$ be the associated effective resistance, then $R\neq R'$.  
\end{proposition}

\noindent\textbf{Remark.} In some contents (see for example Chapter 9 of \cite{LP}), the wired effective resistances and the free effective resistances are defined on an infinite electrical network. Our effective resistances are the free effective resistances. This will not affect our further discussions. \vspace{0.15cm}

Now, we consider the general settings.

\noindent\emph{\textbf{Resistance metric}}. Let $K$ be a set. $R\in l(K\times K)$ is called a resistance metric on $K$ if for any finite subset $V$ of $K$, there is a resistance form $\mcD_V$ on $V$ so that $R|_{V\times V}$ is the effective resistance associated with $\mcD_V$, i.e.
\[R(x,y)=\big(\inf\{\mcD_V(f,f):f\in l(V),f(x)=0,f(y)=1\}\big)^{-1},\ \forall x\neq y\in V.\]
We call $(K,R)$ a \textbf{\emph{resistance space}} if $R$ is a resistance metric on $K$. \vspace{0.15cm}

\noindent\textbf{\emph{Resistance form}}. Let $K$ be a set, and $l(K)$ be the space of real-valued functions on $K$. A pair $(\mathcal{E},\mathcal F)$ is called a resistance form on $K$ if it satisfies the following  conditions:

\noindent (RF1). \textit{$\mathcal{F}$ is a linear subspace of $l(K)$ containing constants and $\mathcal{E}$ is a nonnegative symmetric bilinear form on $\mathcal F$; $\mathcal{E}(f,f)=0$ if and only if $f$ is constant on $K$.}

\noindent (RF2). \textit{Let $\sim$ be an equivalent relation on $\mathcal{F}$ defined by $f\sim g$ if and only if $f-g$ is constant on $K$. Then $(\mathcal{F}/\sim, \mathcal{E})$ is a Hilbert space.}

\noindent (RF3). \textit{For any finite subset $V\subset K$ and for any $g\in l(V)$, there exists $f\in \mathcal{F}$ such that $f|_V=g$.}

\noindent (RF4). \textit{For any $x,y\in K$, $R(x,y):=\big(\inf\{\mcE(f,f):f\in \mcF,f(x)=0,f(y)=1\}\big)^{-1}$ is finite.}

\noindent (RF5). (Markov property) \textit{If $u\in \mathcal{F}$, then $\bar{f}={ \min\{\max\{f,0\}, 1\}}\in \mathcal{F}$ and $\mathcal{E}(\bar{f},\bar{f})\leq \mathcal{E}(f,f)$.}\vspace{0.15cm}

It has been shown in \cite{ki3} that any separable resistance space admits a natural resistance form on it.

\begin{theorem}[\cite{ki3}]\label{thm32}
	Let $(K,R)$ be a complete separable resistance space. There is a unique resistance form $(\mcE,\mcF)$ associated with $(K,R)$ such that 
	\[R(x,y)=\big(\inf\{\mcE(f,f):f\in \mcF,f(x)=0,f(y)=1\}\big)^{-1},\qquad\forall x\neq y.\]
	In addition, we have the following properties. 
	
	(a). Let $V$ be a finite subset of $K$, and let $\mcD_V$ be the resistance form on $V$ associated with $R$, then  
	\[\mcD_V(f,f)=\min\{\mcE(g,g):g\in \mcF, g|_V=f\},\qquad \forall f\in l(V),\]
	where the minimal on the right side is achieved by a unique function, denoted by $h_f$ for convenience. We call $h_f$ the \textbf{harmonic extension} of $f$. 
	
	(b). If $(K,R)$ is locally compact and $\mu$ is a Radon measure on $K$ with full support, then $(\mcE,\mcF\cap L^2(K,\mu))$ is a Dirichlet form on $L^2(K,\mu)$, where $L^2(K,\mu)$ is the space of $L^2$ integrable functions with respect to $\mu$. 
\end{theorem}

\subsection{Stochastic processes on the resistance space} From now on, we assume that $(K,R)$ is a proper separable resistance space. \emph{We will always use $(\mcE,\mcF)$ to denote the associated resistance form. Also, we fix a Radon measure $\mu$ with full support on $K$ (in particular, the choice of $\mu$ will affect the time \cite{CHK}, while the path will not be influenced, so $\mu$ is of less interest in this paper).}\vspace{0.15cm}

\noindent\textbf{\emph{Proper}.} A metric space is called proper if  every closed bounded subset of $K$ is compact. \vspace{0.15cm}

\emph{Also, \textbf{in the rest of this section, and in Section 4}, we will use the notation $B_R(x,\rho)=\{y\in K:R(x,y)<\rho\}$ to denote the open ball in $K$ with respect to the resistance metric $R$; we will use the notation $\text{cl}(U)$ for the closure of $U$ in $(K,R)$; we will use the notation $\partial U$ for the boundary of $U$ in $(K,R)$. Finally, for each $f\in l(K)$, we denote the (closed) support of $f$ by $\text{supp}(f)=\text{cl}(\{x\in K:f(x)\neq 0\})$.}\vspace{0.15cm}

To have a well defined stochastic process on $(K,R)$, we need the additional assumption named `Regular property'. Also, we are interested in continuous paths, so we need the `local property' as well. Here we state both properties the setting of resistance forms. Also read \cite{CF,FOT} for stories in the more general setting of Dirichlet forms.\vspace{0.15cm}

\noindent\textbf{Regular.} We say $(\mcE,\mcF)$ is regular if $\mcF$ is dense in $C_0(K)$ with respect to the supremum norm, where $C_0(K)=\{f\in C(K):\lim\limits_{x\to\infty}f(x)=0\}$.

\noindent\textbf{Local.}  We say $(\mcE,\mcF)$ is local if $\mcE(f,g)=0$ for any compactly supported functions $f,g\in \mcF$ such that $\text{supp}(f)\cap \text{supp}(g)=\emptyset$. \vspace{0.15cm}

By the arguments in Section 2.4 of \cite{ki3}, if $(\mcE,\mcF)$ is local regular, then $(\mcE,\mcF_0)$ is a local regular Dirichlet form on $L^2(K,\mu)$, where $\mcF_0$ is the closure of $\mcF\cap C_0(K)$ with respect to the $\mcE_1$ norm (the $\mcE_1$ inner product is defined as $\mcE_1(f,f)=\mcE(f,f)+\|f\|^2_{L^2(K,\mu)}$). Hence, by Theorem 7.2.1 and Theorem 4.5.1 of \cite{FOT}, there is a diffusion process (Hunt process with continuous sample paths) $(\Omega,X_t,\mathbb{P}_x)$ associated with $(\mcE,\mcF)$ on $K$. In addition, the process is unique by Theorem 4.2.8 of \cite{FOT}, noticing that each point has positive capacity in our setting.

In the remaining of this section, we consider the properties of entry times/hitting times. \vspace{0.15cm}

\noindent\emph{\textbf{Entry time}.} For each Borel $A\subset K$, we define the entry time of $A$ as 
\[\tau_A:=\inf\{t\geq 0:X_t\in A\},\]
where we still admit the setting $\inf\emptyset=\infty$.
 
For each \textit{finite subset} $V$ of $K$, we define $\tau^{(0)}_V:=\tau_V$, and iteratively let  
\[
\tau_V^{(n)}:=\inf\big\{t\geq\tau^{(n-1)}_V:X_t\in V\setminus \{X_{\tau^{(n-1)}_V}\}\big\},\qquad\forall n\geq 1.
\] 
\noindent\emph{\textbf{Shift mapping}.} Guaranteed by the definition of the Hunt processes \cite{FOT}, there are shift mappings $\theta_s:\Omega\to \Omega$, $s\in [0,\infty)$ such that $X_t\circ \theta_s=X_{t+s}$ for any $t,s\in [0,\infty)$.\vspace{0.15cm}

In the rest of this paper, we will focus on a simple case that hitting is guaranteed: let $U$ be a bounded open subset of $K$, and let $A=K\setminus U$, then $\mathbb{P}_x(\tau_A<\infty)=1$, for each $x\in U$. Now, we consider a random walk (Markov chain) on $V$ defined as the trace of $X$, and it is `killed' at $A$.

\begin{definition}\label{def33}
Let $A=K\setminus U$ for some bounded open $U\subsetneq K$, and let $V\subset U$ be a finite subset. We define
\[
X^{[V,A]}_n=
\begin{cases}
	X_{\tau_V^{(n)}},&\text{ if }\tau_V^{(n)}<\tau_A,\\
	\Delta,&\text{ if }\tau_V^{(n)}\geq \tau_A,
\end{cases}
\]
for $n\geq 0$, where $\Delta$ is an isolated point (an absorbing state that represents $A$). In addition, for convenience, we define 
\[\xi=\xi_V=\xi_V(\omega):=\min\{n\geq 0: X^{[V,A]}_n=\Delta\}-1.\]
\end{definition}

It is not hard to see that $\big(\Omega,X^{[V,A]}_n,\mathbb{P}_y\big)$ is a Markov chain on $V\cup\{\Delta\}$ (if $X_0=\Delta$, we simply set $X_1=\Delta=X_2=\cdots$). 

Finally, we end this section with some estimates of hitting times.

\begin{definition}\label{def34}
	Let $A,A'$ be disjoint closed subsets of $K$. We define the effective resistance between $A,A'$ as  
	\[R(A,A')=\big(\inf\{\mathcal{E}(f,f):f\in \mathcal{F},f|_A=0,f|_{A'}=1\}\big)^{-1}.\]
\end{definition}

\begin{lemma}\label{lemma35}
	Let $A=K\setminus U$ for some bounded open $U\subsetneq K$, and let $x,y\in U$. Then 
	\[\mathbb{P}_x(\tau_y<\tau_A)\geq 1-\frac{R(x,y)}{R(x,A)-R(x,y)}.\]
\end{lemma}
\begin{proof}
	Let $f\in l(A\cup\{y\})$ be defined as $f(y)=1$ and $f(z)=0,\forall z\in A$. Let $h_f\in \mcF$ be the harmonic extension of $f$, i.e. $h_f$ is the unique extension of $f$ such that $\mcE(h_f,h_f)=\min\{\mcE(g,g):g\in \mcF,g|_{A\cup\{y\}}=f\}$. We have (see, for example, Proposition 2.5 (a) of \cite{BBKT}) that
	\[h_f(x)=\mathbb{P}_x(\tau_y<\tau_A).\] 
	In addition, $g_U(x,y)=h_f(x)\cdot R(y,A)$ is the Green's function on $U$, i.e. $\mcE\big(g_U(\cdot,y),u\big)=u(y)$ for all $u\in \mcF$ such that $u|_A=0$. By Theorem 4.3 of \cite{ki5}, we know $g_U(y,y)-g_U(x,y)\leq R(x,y)$, hence $h_f(y)-h_f(x)\leq R(x,y)/R(y,A)$, which implies $\mathbb{P}_x(\tau_y<\tau_A)\geq 1-R(x,y)/R(y,A)$.
	
	The lemma follows immediately, noticing that $R(y,A)\geq R(x,A)-R(x,y)$. In fact, Theorem 4.3 of \cite{ki5} implies $R(y,A)\geq R(x,A)-R_A(x,y)$, where $R_A(x,y)=\big(\inf\{\mcE(f,f):f(x)=1,f(y)=0,f\text{ is a constant on }A\}\big)^{-1}\leq R(x,y)$. 
\end{proof}

The following result is a consequence of Lemma 18.1 of \cite{ki5}. 
\begin{lemma}[\cite{ki5}]\label{lemma36}
Let $A=K\setminus U$ for some bounded open $U\subsetneq K$. Then
	\[\frac{1}{2}\mu(U')R(x,A)\leq \mathbb{E}_x(\tau_A)\leq \mu(U)R(x,A),\]
where $U'=\{y\in U:R(x,y)<\frac12 R(x,A)\}$. 
\end{lemma}

\section{Loop-erased random paths on resistance spaces}
In this section, using Theorem \ref{th1}, we prove Theorem \ref{th2}, which defines LERP on a resistance space $(K,R)$, as the limit of the LE of Markov chains on approximating sequences of finite subsets of $K$. We will see that the LERP can also be viewed as the paths obtained by erasing loops from the sample paths of the diffusion processes in a reasonable order (Theorem \ref{thm49}). 

The proof will be divided into three parts. In the first step, we introduce the partial loop erasure (PLE) of continuous paths, and develop a coupling of the distributions using Theorem \ref{th1}. In the second step, we apply the idea of \cite{S}, with some modifications, to show that the discrete paths are not very different from the continuous paths, and to show that the limit distribution supports on simple paths. Finally, we show that the limit distribution is independent of the approximating sequences of finite sets. 

Before proceeding to the proof, we explain some terminologies formally.  As in the classical cases \cite{Ko,LSW,S}, we will consider the weak convergence of probability measures on the space of compact subsets of $(K,R)$ endowed with the Hausdorff metric.\vspace{0.15cm}

\noindent\textbf{\emph{Hausdorff metric}.} Let $(K,R)$ be a proper metric space, and $A,A'\subset K$ be two compact subsets, we define the Hausdorff metric between $A,A'$ as 
\[
d_H(A,A'):=\inf\{\rho>0:A\subset B_R(A',\rho),A'\subset B_R(A
,\rho)\},
\]
where $B_R(A,\rho)=\bigcup_{x\in A}B_R(x,\rho)$. For convenience, we use the notation $\mathcal{K}=\{A\subset K:A\text{ is compact}\}$, then $(\mathcal{K},d_H)$ is a proper metric space. \vspace{0.2cm}

\noindent\textbf{\emph{Weak convergence}.} Let $\nu_n,n\geq 1$ and $\nu$ be probability measures on $(\mathcal{K},d_H)$. We say $\nu_n$ converges weakly to $\nu$, and write $\nu_n\Rightarrow \nu$ if 
\[\int_\mathcal{K}f(A)\nu_n(dA)\to \int_\mathcal{K}f(A)\nu(dA),\text{ as }n\to\infty,\quad \forall f\in C_b(\mathcal{K}),\]
where $C_b(\mathcal{K})$ is the space of bounded continuous functions on $(\mathcal{K},d_H)$. It is well known that $\nu_n\Rightarrow \nu$ if and only if $\pi(\nu_n,\nu)\to 0$, where $\pi$ is the Prokhorov metric (see Theorem 6.8 of \cite{Bi} for example).

\noindent\textbf{\emph{The Prokhorov metric}.} Let $\nu_1,\nu_2$ be two probability measures on $(\mathcal{K},d_H)$, the Prokhorov metric $\pi(\nu_1,\nu_2)$ is defined as 
\[
\begin{aligned}
	\pi(\nu_1,\nu_2)=\inf\{\varepsilon>0:\nu_2(\mathcal{A})\leq\nu_1(\mathcal{A}^\varepsilon)&+\varepsilon,\  \nu_1(\mathcal{A})\leq\nu_2(\mathcal{A}^\varepsilon)+\varepsilon,\\
	&\text{ for any Borel subset }\mathcal{A}\text{ of }\mathcal{K}\},
\end{aligned}
\]
where $\mathcal{A}^\varepsilon=\{A'\in \mathcal{K}:d_H(A,A')<\varepsilon\}$. 
\vspace{0.2cm}

\noindent\textbf{\emph{Image}.} We will view paths as subsets of $K$. For clearance, for any mapping $f:K_1\to K_2$, we denote the image of $f$ by $\Im f=\Im(f)=:\{f(x):x\in K_1\}$. \vspace{0.2cm} 

\noindent\textbf{\emph{Simple path}.} We say a continuous path $\gamma:[0,s]\to K$ is simple if $\gamma(t_1)\neq \gamma(t_2),\forall 0\leq t_1<t_2\leq s$. We also call $\Im(\gamma)$ a simple path. \vspace{0.2cm} 

\begin{comment}
Now we introduce the second main theorem of this paper. Recall that $X^{[V,A]}|_{[0,\xi]}$ is defined in Definition \ref{def33} as the trace of $X_t$ on $V$ before hitting $A$, where $\xi$ is the lifetime.  

\begin{theorem}[LERP on resistance spaces]\label{thm41}
	Let $(K,R)$ be a proper local regular resistance space, let $\mu$ be a Radon measure on $K$ with full support, let $(\Omega,X_t,\mathbb{P}_x)$ be the associated diffusion process, and let $A=K\setminus U$ for some bounded open $U\subsetneq K$. Then, for each $x\in U$, there exists a probability measure $\nu_{x,A}$ on $(\mathcal{K},d_H)$ supported on simple paths connecting $x$ and $A$ such that
	\[\mathbb{P}_x\big(\Im\mfL(X^{[V_m,A]}|_{[0,\xi]})\in \bullet\big)\Rightarrow \nu_{x,A},\text{ as }m\to\infty\]
	 on $(\mathcal{K},d_H)$ for any expanding sequence $V_1\subsetneq V_2\subsetneq V_3\subsetneq \cdots$ of finite subsets of $U$ such that $V_*=\bigcup_{m=0}^\infty V_m$ is dense in $U$. 
\end{theorem}
\end{comment}

\subsection{A natural coupling via PLE} One can naturally define the PLE of a continuous curve. For convenience, we first introduce some notations.

\begin{definition}\label{def42}
	(a). Let $\gamma:[0,s]\to K$ be a continuous path. With a little abuse of notation, we write $\tau^{(0)}_V(\gamma)=\tau_V(\gamma):=\inf\{t\geq 0:X_t\in V\}$, and iteratively for $n\geq 1$,   
	\[\tau_V^{(n)}=\tau_V^{(n)}(\gamma):=\inf\big\{\tau_V^{(n-1)}(\gamma)<t\leq s:\gamma(t)\in V\setminus \{\gamma(\tau^{(n-1)}_V)\}\big\}.\]	
	In addition, for convenience, we write
	\[\xi=\xi_V=\xi_V(\gamma):=\min\{n\geq 0:\tau_V^{(n)}=\infty\}-1.\]
	
	(b). Let $\gamma_1:[s_1,t_1]\to K,\gamma_2:[s_2,t_2]\to K$ and assume $\gamma_1(t_1)=\gamma_2(s_2)$. We define the concatenation $\gamma_1\oplus\gamma_2:[0,t_1+t_2-s_1-s_2]\to K$ as
	\[\gamma_1\oplus \gamma_2(t)=
	\begin{cases}
		\gamma_1(t+s_1),&\text{ if }0\leq t\leq t_1-s_1,\\
		\gamma_2(t-t_1+s_1+s_2),&\text{ if }t_1-s_1<t\leq t_1+t_2-s_1-s_2.
	\end{cases}\]
\end{definition}

\noindent\textbf{\emph{PLE of a continuous curve.}} Let $\gamma:[0,s]\to K$ be a continuous curve, and let $V_1\subset V_2$ be finite subsets of $K$. Let 
\[\bw=\big(\gamma(\tau_{V_2}^{(0)}),\gamma(\tau_{V_2}^{(1)}),\cdots, \gamma(\tau_{V_2}^{(\xi)})\big),\text{ if }\xi_{V_2}\geq 0,\]
and let $\mfI_{V_1}(\bw)=(n_0,n_1,\cdots,n_\iota)$ if $\xi\geq 0$ (in particular, $n_0=0$). 

Then we define the PLE (associated with $V_1,V_2$) of $\gamma$ as 
\[
\mfL_{V_2,V_1}(\gamma)=
	\gamma|_{[0,\tau_{V_2}^{(0)}]}\oplus \gamma|_{[\tau_{V_2}^{(n_1-1)},\tau_{V_2}^{(n_1)}]}\oplus\cdots\oplus\gamma|_{[\tau_{V_2}^{(n_{\iota-1})},\tau_{V_2}^{(n_\iota)}]}\oplus \gamma|_{[\tau_{V_2}^{(n_\iota)},s]}
\]
if $\xi_{V_2}(\gamma)\geq 0$, and simply let $\mfL_{V_2,V_1}(\gamma)=\gamma$ if $\xi_{V_2}(\gamma)=-1$. 

In particular, we will write $\mfL_{V_1}=\mfL_{V_1,V_1}$ for short.

\begin{proposition}\label{prop43}
	Assume all the settings of Theorem \ref{th2}, and let $V_1\subsetneq V_2\subsetneq V_3\subsetneq \cdots$ be a sequence of finite subsets of $U$. Then, for each $x\in U$ and $m\geq 1$, we have
	\[
	\mathbb{P}_x\big(\mfL_{V_m}(X|_{[0,\tau_A]})\in \bullet\big)
	=\mathbb{P}_x\big(\mfL_{V_m}\circ \mfL_{V_m,V_{m-1}}\circ \cdots\circ\mfL_{V_3,V_2}\circ\mfL_{V_2,V_1}(X|_{[0,\tau_A]})\in \bullet\big),
	\]
	where $X|_{[0,\tau_A]}:[0,\tau_A]\to K$ is viewed a random path.  
\end{proposition}
\begin{proof}
By applying Theorem \ref{th1} and using the strong Markov property as in Lemma \ref{lemma28}, we can see $\mathbb{P}_x\big(\mfL_{V_2}(X|_{[0,\tau_A]})\in \bullet\big)
=\mathbb{P}_x\big(\mfL_{V_2}\circ\mfL_{V_2,V_1}(X|_{[0,\tau_A]})\in \bullet\big)$. Next, noticing that $\mfL_{V_{k+1},V_k}\circ\mfL_{V_k}=\mfL_{V_{k+1},V_k}$ for $k=2,3,\cdots m-1$, we can apply induction hypothesis to complete the proof.	
\end{proof}

\begin{definition}\label{def44}
Assume all the settings of Theorem \ref{th2}, and let $V_1\subsetneq V_2\subsetneq V_3\subsetneq \cdots$ be a sequence of finite subsets of $U$ that satisfies $V_*=\bigcup_{m=0}^\infty V_m$ is dense in $U$. We define (which in fact depends on the sequence, not merely on $V_*$) 
\[
\Im\mfL_{V_*}(X|_{[0,\tau_A]})=\bigcap_{m=1}^\infty \Im\mfL_{V_m}\circ \mfL_{V_m,V_{m-1}} \circ\cdots\circ\mfL_{V_2,V_1}(X|_{[0,\tau_A]}).
\]
In addition, we write $\nu^*_{x,A}=\mathbb{P}_x\big(\Im\mfL_{V_*}(X|_{[0,\tau_A]})\in \bullet\big)$ for each $x\in U$.
\end{definition}

\begin{corollary}\label{coro45}
Assuming all the settings of Definition \ref{def44}. Then
	\[\mathbb{P}_x\big(\Im\mfL_{V_m}(X|_{[0,\tau_A]})\in \bullet\big)\Rightarrow \nu_{x,A}^*, \text{ as }m\to\infty\]
	for each $x\in U$ on $(\mathcal{K},d_H)$.
\end{corollary}
\begin{proof}
	Noticing that $\Im\mfL_{V_m}\circ \mfL_{V_m,V_{m-1}}\circ \cdots\circ\mfL_{V_2,V_1}(X|_{[0,\tau_A]}),m\geq 1$ forms a shrinking sequence of connected compact sets, $\Im\mfL_{V_*}(X|_{[0,\tau_A]})$ is also connected compact, and in addition
	\[\lim\limits_{m\to\infty}d_H\big(\Im\mfL_{V_*}(X|_{[0,\tau_A]}),\Im\mfL_{V_m}\circ \mfL_{V_m,V_{m-1}}\circ \cdots\circ\mfL_{V_2,V_1}(X|_{[0,\tau_A]})\big)=0.\]
	The corollary follows immediately from Proposition \ref{prop43}. 
\end{proof}

\subsection{Loops and long jumps} In this section, for a fixed sequence $V_1\subsetneq V_2\subsetneq \cdots$, we will show that the limit distribution in Theorem \ref{th2}, denoted by $\nu_{x,A}$, is in fact $\nu^*_{x,A}$ introduced in Definition \ref{def44}. To achieve this, our main goal in this part consists of two parts: 
\begin{itemize}
	\item show that $\nu^*_{x,A}$ supports on simple curves;
	
	\item show that $\Im\mfL(X^{[V_m,A]}|_{[0,\xi]})$ and $\Im\mfL_{V_m}(X|_{[0,\tau_A]})$ are not so different. 
\end{itemize} 

Some key arguments in this part are inspired by O. Schramm \cite{S}. There are some differences in our setting compared with \cite{S}: first, since we are looking at strongly recurrent diffusions, we can erase loops directly from the continuous paths, so we only need to take care of loops instead of quasi-loops; second, we need to take care of possible long jumps, since the Markov chains we consider here are not defined on good graphs in general.\vspace{0.15cm}

Let $V$ be a finite subset of $K$, $z_0\in K$ and $\rho>0$. 

\noindent\textbf{$(\rho,V)$-Long-jumps.} A $(z_0,\rho,V)$-long-jump in a path $\gamma$ is a subarc $\gamma|_{[s_1,s_2]}$ such that $\gamma(s_1)\in B_R(z_0,\rho)$, $\gamma(s_2)\in K\setminus B_R(z_0,2\rho)$ and  $\Im(\gamma|_{[s_1,s_2]})\cap V=\emptyset$. Let $\mathscr{J}(z_0,\rho,V)$ denote the set of paths that have a $(z_0,\rho,V)$-long-jump.

In addition, a $(\rho,V)$-long-jump in a path $\gamma$ is a subarc $\gamma|_{[s_1,s_2]}$ such that $R\big(\gamma(s_1),\gamma(s_2)\big)\geq \rho$ and $\Im(\gamma|_{[s_1,s_2]})\cap V=\emptyset$. Let $\mathscr{J}(\rho,V)$ denote the set of paths that have a $(\rho,V)$-long-jump.

\noindent\textbf{\emph{$\rho$-Loops}.} A $(z_0,\rho)$-loop in a path $\gamma$ is a subarc $\gamma|_{[s_1,s_2]}$ such that $\gamma(s_1)=\gamma(s_2)\in B_R(z_0,\rho)$ and $\Im(\gamma|_{[s_1,s_2]})\nsubseteq B_R(z_0,2\rho)$. Let $\mathscr{L}(z_0,\rho)$ denote the set of paths that have a $(z_0,\rho)$-loop.

In addition, a $\rho$-loop is a is a subarc $\gamma|_{[s_1,s_2]}$ of $\gamma$ such that $\gamma(s_1)=\gamma(s_2)$ and $\Im(\gamma|_{[s_1,s_2]})\nsubseteq B_R\big(\gamma(s_1),\rho\big)$. Let $\mathscr{L}(\rho)$ denote the set of paths that have a $\rho$-loop.

\begin{lemma}\label{lemma46}
Assume all the settings of Definition \ref{def44}. Let $x,z_0\in U$ and $\rho>0$. Also, we assume $B_R(z_0,2\rho)\subset U$.
	
(a). $\lim\limits_{m\to\infty}\mathbb{P}_x\big(\mfL_{V_m}(X|_{[0,\tau_A]})\in\mathscr{J}(z_0,\rho,V_m)\big)=0$. 
	
(b). $\lim\limits_{m\to\infty}\mathbb{P}_x\big(\mfL_{V_m}(X|_{[0,\tau_A]})\in\mathscr{J}(\rho,V_m)\big)=0.$ %For any fixed $\rho>0$, and any open $U'$ such that $R(U',A)>0$, we have 
	
(c). $\lim\limits_{m\to\infty}\mathbb{P}_x\big(\mfL_{V_m}(X|_{[0,\tau_A]})\in\mathscr{L}(z_0,\rho)\big)=0$. 
	
(d). $\lim\limits_{m\to\infty}\mathbb{P}_x\big(\mfL_{V_m}(X|_{[0,\tau_A]})\in\mathscr{L}(\rho)\big)=0$. 
\end{lemma}
\begin{proof}
	For short, we let $B_1=B_R(z_0,\rho)$ and $B_2=B_R(z_0,2\rho)$. Also, define $s_1=\inf\{0\leq t\leq \tau_{A}:X_t\in B_1\}$, and let $t_1=\inf\{s_1\leq t\leq \tau_{A}:X_t\notin B_2\}$, and inductively $s_i=\inf\{t_{j-1}\leq t\leq \tau_{A}:X_t\in B_1\}$ and $t_i=\inf\{s_i\leq t\leq \tau_{A}:X_t\notin B_2\}$ for $i>1$, where we still let $\inf\emptyset=\infty$. Let $\iota=\max\{i\geq 1:s_i<\infty\}$ if $\{i\geq 1:s_i<\infty\}\neq \emptyset$, and otherwise let $\iota=0$. 
	
	Noticing that $\inf_{y\in \text{cl}(B_1)}\mathbb{E}_y(\tau_{K\setminus B_2})>0$ by Lemma \ref{lemma36}, we have 
	\begin{align*}
		R(x,A)\mu(U)&\geq \mathbb{E}_x(\tau_A)\geq \mathbb{E}_x\big(\sum_{i=1}^{\iota} (t_i-s_i)\big)=\sum_{i=1}^\infty\mathbb{E}_x(t_i-s_i|\iota\geq i)\mathbb{P}_x(\iota\geq i)
		\\&\geq \inf_{y\in \text{cl}(B_1)}\mathbb{E}_y(\tau_{K\setminus B_2})\cdot \sum_{i=1}^\infty\mathbb{P}_x(\iota\geq i)=\inf_{y\in \text{cl}(B_1)}\mathbb{E}_y(\tau_{K\setminus B_2})\cdot\mathbb{E}_x(\iota).
	\end{align*}
	Hence $\mathbb{E}_x(\iota)<\infty$.

	(a). Let's first show that $\mathbb{P}_y\big(X|_{[0,\tau_{K\setminus B_2}]}\in\mathscr{J}(z_0,\rho,V_m)\big)\to 0$ uniformly for $y\in \text{cl}(B_1)$. For a fixed $\varepsilon>0$, by Lemma \ref{lemma35}, we can choose $A'=B_R(K\setminus B_2,\delta)$ for some $\delta$ small enough so that $\mathbb{P}_z(\tau_{K\setminus B_2}<\tau_{B_1})>1-\epsilon$ for any $z\in \partial A'$. In addition, for $m$ large enough, we also have  $\mathbb{P}_z(\tau_{V_m}<\tau_{K\setminus B_2})>1-\epsilon$ for any $z\in \partial A'$ by Lemma \ref{lemma35}. Hence, for large enough $m$, 
	\begin{align*}
		\mathbb{P}_y\big(X|_{[0,\tau_{K\setminus B_2}]}\notin\mathscr{J}(z_0,\rho,V_m)\big)
		\geq&\mathbb{P}_y\big(\tau_{V_m}\circ \theta_{\tau_{A'}}<\tau_{K\setminus B_2}\circ\theta_{\tau_{A'}}<\tau_{B_1}\circ \theta_{\tau_{A'}}\big)\\
		%=&\int_{\partial A'}\mathbb{P}_y(X_{\tau_{A'}}\in dz\big)\cdot\mathbb{P}_z(\tau_{V_m}<\tau_{K\setminus B_2}<\tau_{B_1})\\
		> &1-2\epsilon
	\end{align*}
    holds for any $y\in \text{cl}(B_1)$. Finally, one can see that 
	\begin{align*}
		&\mathbb{P}_x\big(\mfL_{V_m}(X|_{[0,\tau_A]})\in\mathscr{J}(z_0,\rho,V_m)\big)
		\leq \mathbb{P}_x\big(X|_{[0,\tau_A]}\in\mathscr{J}(z_0,\rho,V_m)\big)\\
	   =&\mathbb{P}_x\big(X|_{[s_i,t_i]}\in\mathscr{J}(z_0,\rho,V_m)\text{ for some }1\leq i\leq \iota\big)\\
   \leq &\sum_{i=1}^\infty \mathbb{P}_x\big(X|_{[s_i,t_i]}\in\mathscr{J}(z_0,\rho,V_m)|\iota\geq i\big)\cdot\mathbb{P}_x(\iota\geq i)\\
   \leq &\mathbb{E}_x(\iota)\sup_{y\in \text{cl}(B_1)}\mathbb{P}_y\big(X|_{[0,\tau_{K\setminus B_2}]}\in\mathscr{J}(z_0,\rho,V_m)\big)
	\end{align*}
	converges to $0$ as $m\to 0$ since $\mathbb{E}_x(\iota)<\infty$. 
	
	(b). Fix an arbitrary $\varepsilon>0$, then we choose small enough $\delta<1/3$ and $U'$ of the form $U'=\bigcup_{j=1}^MB_R(z_j,\delta\rho)$, such that 
	\[
	\mathbb{P}_y(\tau_A<\tau_{\partial B_R(y,\rho/2)})>1-\varepsilon,\qquad\forall y\in \partial U',
	\]
	and $\bigcup_{i=1}^MB_R(z_j,2\delta\rho)\subset U$. This is feasible by Lemma \ref{lemma35}. Then one can see that 
	\begin{align*}
	\mathbb{P}_x\big(\mfL_{V_m}(X|_{[0,\tau_A]})\in \mathscr{J}(\rho,V_m)\big)
	\leq&\sum_{j=1}^M\mathbb{P}_x\big(\mfL_{V_m}(X|_{[0,\tau_A]})\in\mathscr{J}(z_j,\frac{\rho}{3},V_m)\big)\\
	&+\mathbb{P}_x(\tau_A\circ\theta_{\tau_{\partial U'}}>\tau_{\partial B_R(X_{\tau_{\partial U'}},\frac{\rho}{2})}\circ \theta_{\tau_{\partial U'}})\\
	\leq&\sum_{j=1}^M\mathbb{P}_x\big(\mfL_{V_m}(X|_{[0,\tau_A]})\in\mathscr{J}(z_j,\delta\rho,V_m)\big)+\varepsilon,
	\end{align*}
	which is smaller than $2\varepsilon$ when $m$ is large enough by (a).
	
	(c). In fact, by a same argument as in (b), one also have $\lim\limits_{m\to\infty}\mathbb{P}_x\big(X|_{[0,\tau_A]}\in\mathscr{J}(\rho,V_m)\big)=0$. The proof of (c) now is very similar to that of \cite{S} Lemma 3.4. We let $\mathscr{L}_{i,m}$ denote the event that $\mfL_{V_m}(X|_{[0,t_i]})\in \mathscr{L}(z_0,\rho)$, and $\neg \mathscr{L}_{i,m}$ is the event that $\mfL_{V_m}(X|_{[0,t_i]})\notin \mathscr{L}(z_0,\rho)$. One can see that for each $i\geq 2$, 
	\[
	\lim\limits_{m\to\infty}\mathbb{P}(\mathscr{L}_{i,m},\neg \mathscr{L}_{i-1,m},\iota\geq i)=0.
	\]
	In fact, $(\neg\mathscr{L}_{i-1,m})\cap\mathscr{L}_{i,m}$ implies that $X|_{[s_i,t_i]}$ hits one of the components, denoted by $\Im\gamma'$ for convenience, of $\Im\mfL_{V_m}(X|_{[0,t_{i-1}]})\cap B_2$ that is connected to $B_1$, while $X|_{[s_i,t_i]}$ does not hit the $V_m\cap \Im\gamma'$. On the other hand, as pointed out at the beginning, for any small $\varepsilon$, when $m$ is large enough, we can ignore the case that $\mfL_{V_m}(X|_{[0,t_{i-1}]})$ has a $(\varepsilon,V_m)$-long-jump, whose probability is very small. Hence, $X|_{[s_i,t_i]}$ hits to within distance $\varepsilon$ to $\Im\gamma'\cap V_m$, but does not hit $\Im\gamma'\cap V_m$, whose probability is small. Also noticing that there are at most $i-1$ such components, we see the claim holds. Hence,
	\begin{align*}
		\mathbb{P}_x\big(\mfL_{V_m}(X|_{[0,\tau_A]})\in\mathscr{L}(z_0,\rho,\varepsilon)\big)
		&\leq \sum_{i=2}^M\mathbb{P}(\mathscr{L}_{i,m},\neg\mathscr{L}_{i-1,m},\iota\geq i)+\mbP_x(\iota>M)\\
		&\leq \sum_{i=2}^M\mathbb{P}(\mathscr{L}_{i,m},\neg \mathscr{L}_{i-1,m},\iota\geq i)+\mathbb{E}_x(\iota)/M.
	\end{align*}
	Let $M\to\infty$ slowly as $m\to\infty$, one then see (c) holds. 
	
	(d). Fix an arbitrary $\varepsilon>0$, and choose small enough $\delta<1/3$ and $U'$ of the form $U'=\bigcup_{i=1}^MB_R(z_j,\delta\rho)$, such that 
	\[
	\mathbb{P}_y(\tau_A\leq\tau_{\partial B_R(y,\rho/2)})\geq 1-\varepsilon,\qquad\forall y\in \partial U',
	\]
	and $\bigcup_{i=1}^MB_R(z_j,2\delta\rho)\subset U$. Then one can see that 
	\begin{align*}
		\mathbb{P}_x\big(\mfL_{V_m}(X|_{[0,\tau_A]})\in \mathscr{L}(\rho)\big)\leq& \sum_{j=1}^M\mathbb{P}_x\big(\mfL_{V_m}(X|_{[0,\tau_A]})\in\mathscr{L}(z_j,\rho/3)\big)\\
		&+\mathbb{P}_x(\tau_A\circ\theta_{\tau_{\partial U'}}>\tau_{\partial B_R(X_{\tau_{\partial U'}},\rho/2)}\circ \theta_{\tau_{\partial U'}})\\
		\leq &\sum_{j=1}^M\mathbb{P}_x\big(\mfL_{V_m}(X|_{[0,\tau_A]})\in\mathscr{L}(z_j,\delta\rho)\big)+\varepsilon,
	\end{align*}
	which is smaller than $2\varepsilon$ when $m$ is large enough by (c).
\end{proof}

We almost arrive at Theorem \ref{th2}, except that we haven't seen whether the limit depends on the sequence.

\begin{corollary}\label{coro47}
	Assuming all the settings of Corollary \ref{coro45}, we have 
	\[\mathbb{P}_x\big(\Im\mfL_{V_m}(X^{[V_m,A]}|_{[0,\tau_A]})\in \bullet\big)\Rightarrow \nu^*_{x,A}, \text{ as }m\to\infty,\]
	and $\nu^*_{x,A}$ supports on the set of simple curves. 
\end{corollary}
\begin{proof}
	First, by Lemma \ref{lemma46} (b), for any small $\varepsilon>0$, we have 
	\[
	\mathbb{P}_x\Big(d_H\big(\Im\mfL_{V_m}(X^{[V_m,A]}|_{[0,\tau_A]}),\Im\mfL_{V_m}(X|_{[0,\tau_A]})\big)<\varepsilon\Big)>1-\varepsilon.
	\]
	for any large enough $m$. Hence, $\mathbb{P}_x\big(\Im\mfL_{V_m}(X^{[V_m,A]}|_{[0,\tau_A]})\in \bullet\big)\Rightarrow \nu^*_{x,A}$ by Corollary \ref{coro45}. 
	
	It remains to show that $\nu^*_{x,A}$ supports on simple paths. For a strict proof, we use the following topological characterization of arc by Janiszewski \cite{J}.
	
	\begin{lemma}[\cite{J}]\label{lemma47}
		Let $I$ be a compact, connected metric space, and let $o_1,o_2\in I$. If for every $x\in K\setminus\{o_1,o_2\}$ the set $I-\{x\}$ is disconnected, then $I$ is homeomorphic to $[0, 1]$.
	\end{lemma}
	
	For convenience, in the following, we write $\gamma_m=\gamma_m(\omega)=\mfL_{V_m}\circ \mfL_{V_m,V_{m-1}}\circ \cdots\circ\mfL_{V_2,V_1}\big(X(\omega)|_{[0,\tau_A]}\big)$.  Let 
	\[\widetilde{\Omega}=\bigcap_{k=1}^\infty\bigcup_{m=1}^\infty\bigcap_{j=m}^\infty\big\{\omega\in\Omega:\gamma_m(\omega)\notin \mathscr{L}(\frac{1}{k})\big\}.\]
	Then, by Lemma \ref{lemma46} (d), we know $\mathbb{P}_x(\widetilde{\Omega})=1$. For each $\omega\in \widetilde{\Omega}$, we can show that $\Im\mfL_{V_*}\big(X|_{[0,\tau_A]}(\omega)\big)$ is a simple path. 
	
	To see this, let's fix $\omega\in \widetilde{\Omega}$ and $z\in \Im\mfL_{V_*}(X|_{[0,\tau_A]}(\omega))\setminus \{X_0(\omega),X_{\tau_A}(\omega)\}$.
	
	Then, for any large $k\geq 1$, one can see that $\Im\gamma_m(\omega)\setminus B_R(z,1/k)$ is disconnected for $m$ large enough, with $X_0(\omega),X_{\tau_A}(\omega)$ belonging to different components. In fact, if $X_0(\omega),X_{\tau_A}(\omega)$ belong to a same component, we can find a loop of diameter at least $\frac{1}{2k}$. To find the loop, we choose $s_0$ such that $\gamma_m(s_0)=z$, and let $s_-=\sup\{s<s_0:\gamma_m(s)\notin B_R(z,1/k)\},s_+=\inf\{s>s_1:\gamma_m(s)\notin B_R(z,1/k)\}$. Then by the assumption that $X_0(\omega),X_{\tau_A}(\omega)$ belong to a same component of $\Im\gamma_m(\omega)\setminus B_R(z,1/k)$, $\Im(\gamma_m|_{[0,s_-]})\cup \Im(\gamma_m|_{[s_+,s]})$ is connected, where $[0,s]$ is the domain of $\gamma$. So we have $\gamma_m(t_1)=\gamma_m(t_2)$ for some $t_1\leq s_-,t_2\geq s_+$, hence there is a loop in $\gamma_m$ of diameter at least $\frac{1}{2k}$. 
	
	Now since $\Im\mfL_{V_*}\big(X|_{[0,\tau_A]}(\omega)\big)\subset \Im\gamma_m(\omega)$, we can see $\mfI\mfL_{V_*}\big(X|_{[0,\tau_A]}(\omega)\big)\setminus B_R(z,1/k)$ is disconnected with $X_0(\omega),X_{\tau_A}(\omega)$ belonging to different components. Since $k$ is arbitrary, it is not hard to see $\Im\mfL_{V_*}(X|_{[0,\tau_A]}(\omega))\setminus \{z\}$ is disconnected. Hence $\Im\mfL_{V_*}(X|_{[0,\tau_A]}(\omega))$ is a simple path by Lemma \ref{lemma47}. 
\end{proof}

\subsection{Independence of the approximating sequences} It remains to show the limit distribution is independent of the approximating sequence. 

\begin{proof}[Proof of Theorem \ref{th2}]
	Let's fix two approximating sequences of finite subsets of $U$, i.e. 
	\[
	\begin{aligned}
		V_{1,1}\subsetneq V_{1,2}\subsetneq \cdots\subsetneq V_{1,m}\subsetneq \cdots ,\\
		V_{2,1}\subsetneq V_{2,2}\subsetneq \cdots\subsetneq V_{2,m}\subsetneq \cdots,
	\end{aligned}
	\]
	and $V_{1,*}=\bigcup_{m=1}^\infty V_{1,m},\ V_{2,*}=\bigcup_{m=1}^\infty V_{2,m}$ are dense in $U$. By Corollary \ref{coro45} and \ref{coro47}, we know there exist $\nu^{1,*}_{x,A}$ and $\nu^{2,*}_{x,A}$ such that for $i=1,2$,
	\begin{align*}
		\mbP_x\big(\Im\mfL_{V_{i,m}}(X^{[V_{i,m}]}|_{[0,\xi]})\in \bullet\big)\Rightarrow \nu^{i,*}_{x,A},\\
		\nu^{i,m}_{x,A}:=\mbP_x\big(\Im\mfL_{V_{i,m}}(X|_{[0,\tau_A]})\in \bullet\big)\Rightarrow \nu^{i,*}_{x,A}.
	\end{align*}
	
	Now we consider a third sequence $V_{3,m},m\geq 1$ defined by $V_{3,m}=V_{1,m}\cup V_{2,m}$, and we write $\nu^{3,m}_{x,A}:=\mbP_x\big(\Im\mfL_{V_{3,m}}(X|_{[0,\tau_A]})\big)$ for any $m\geq 1$. Then, for any $\varepsilon$, by Lemma \ref{lemma46} (d), for any $m$ large enough, we have 
	\[\mathbb{P}_x\big(\mfL_{V_{i,m}}(X|_{[0,\tau_A]})\in \mathscr{L}(\varepsilon)\big)<\varepsilon,\qquad\text{ for } i=1,2.\] 
	As a consequence, for $i=1,2$,
	\[\mathbb{P}_x\Big(d_H\big(\Im\mfL_{V_{i,m}}(X|_{[0,\tau_A]}),\Im\mfL_{V_{3,m}}\circ \mfL_{V_{3,m},V_{i,m}}\circ\mfL_{V_{i,m}}(X|_{[0,\tau_A]})\big)>\varepsilon\Big)<\varepsilon.\]
	Noticing that $\mfL_{V_{3,m}}\circ \mfL_{V_{3,m},V_{i,m}}\circ\mfL_{V_{i,m}}(X|_{[0,\tau_A]})=\mfL_{V_{3,m}}\circ \mfL_{V_{3,m},V_{i,m}}(X|_{[0,\tau_A]})$ has the law $\nu^{3,m}_{x,A}$ for $i=1,2$ by Proposition \ref{prop43}, we have 
	\[\pi(\nu^{1,m}_{x,A},\nu^{3,m}_{x,A})<\varepsilon,\ \pi(\nu^{2,m}_{x,A},\nu^{3,m}_{x,A})<\varepsilon,\]
	hence $\pi(\nu^{1,m}_{x,A},\nu^{2,m}_{x,A})<2\varepsilon$ for any large $m$. This implies $\nu^{1,*}_{x,A}=\nu^{2,*}_{x,A}$. 
\end{proof}

Finally, at the end of this section, we remark that in the celebrated work \cite{LSW}, which shows that the scaling limit of the LERW on a domain of $\mathbb{C}$ exists and is conformal invariant. There, they considered a slightly stronger metric over the space of unparameterized paths: 

\noindent\emph{\textbf{A metric over unparameterized paths}.} consider the metric $d_p(\gamma,\gamma')=\inf \sup_{t\in [0,1]}|\hat{\gamma}(t)-\hat{\gamma}'(t)|$,  where the infimum is over all choices of parameterizations $\hat{\gamma}$ and $\hat{\gamma}'$ in $[0,1]$ of $\gamma$ and $\gamma'$.\vspace{0.15cm}

Since we are erasing loops from $X|_{[0,\tau_A]}$, it is obvious that our result holds for this metric, if we consider the convergence $\mfL_{V_m}\circ\mfL_{V_m,V_{m-1}}\circ\cdots\circ\mfL_{V_2,V_1}\big(X|_{[0,\tau_A]}\big)\to \mfL_{V_*}\big(X|_{[0,\tau_A]}\big)$.

\begin{theorem}\label{thm49}
Assuming all the settings of Theorem \ref{th2}, we have 
	\[\mathbb{P}_x\big(\mfL_{V_m}(X|_{[0,\tau_A]})\in\bullet\big)\Rightarrow \nu_{x,A},\]
holds with respect to the metric over unparameterized paths for any $V_1\subsetneq V_2\subsetneq\cdots$ such that $V_*=\bigcup_{m=1}^\infty V_m$ is dense in $U$.
\end{theorem}

\section{Scaling limits of loop-erased random walks on planar Sierpi\'nski carpet graphs}
We conclude the paper with an application on the Sierpi\'nski carpets ($\mathcal{SC}$), the scaling limit of LERW on $\mathcal{SC}$ graphs (Theorem \ref{th3}). As discussed in the introduction, for simplicity, we focus on the simple case that the path starts at some point $x$ and ends at another point $y$. 

Throughout this section, $d$ is the Euclidean metric on $\mathbb{R}^2$; $d_H$ is the Hausdorff metric between compact subsets of $\mathbb{R}^2$; $B_d(x,\rho)$ is a ball centered at $x$ with radius $\rho$ with respect to the Euclidean metric $d$; $\partial A$ is the boundary set of $A$ as a subset of $\mathbb{R}^2$, and $\text{int}(A)=A\setminus \partial A$. 

Let $K$ be a $\mathcal{SC}$. Then, the resistance metric $R$ on $K$ is equivalent to $d$ in the sense that $C_1d^\gamma(x,y)\leq R(x,y)\leq C_2 d^{\gamma}(x,y),\forall x,y\in K$ for some $\gamma>0$ and constants $C_1,C_2$ (see \cite{BB2,BB3,BB4,CQ,KZ}), so there is no worry about which metric we use. \vspace{0.15cm}

\noindent\textbf{\emph{Sierpi\'nski carpets \cite{BB1}.}} Let $H_0=[0,1]^2\subset \mathbb{R}^2$, and let $k\geq 3$ be fixed. Set 
\[
Q=\big\{[(i_1-1)/k,i_1/k]\times [(i_2-1)/k,i_2/k]: 1\leq i_1,i_2\leq k,\  i_1,i_2\in \mathbb{Z}\big\}.
\]
Let $4k-4\leq N<k^2$, and let $\Psi_l,1\leq l\leq N$ be orientation-preserving affine maps of $H_0$ onto some element of $Q$ (we assume that $\Psi_l(H_0)$ are distinct). Then there exists a unique compact nonvoid set $K\subset H_0$ such that $K=\bigcup_{l=1}^N \Psi_l(K)$. $K$ is called a (planar) Sierpi\'nski carpet if the following holds for $H_1=\bigcup_{l=1}^N\Psi_l(H_0)$:   

\noindent\textbf{(Symmetry)}: $H_1$ is preserved by all the isometries of the unit square $H_0$. 

\noindent\textbf{(Connected)}: $H_1$ is connected.

\noindent\textbf{(Nondiagonality)}: Let $B$ be a cube in $H_0$ that is the union of $4$ distinct elements of $Q$ (So $B$ has side length $2/k$). Then if $\text{int}(H_1\cap B)$ is non-empty, it is connected.

\noindent\textbf{(Borders included)}: $H_1$ contains $\partial H_0$. 

\noindent See Figure \ref{SC} for the standard $\mathcal{SC}$ ($k=3,N=8$).

\begin{figure}[htp]
	\includegraphics[width=3.8cm]{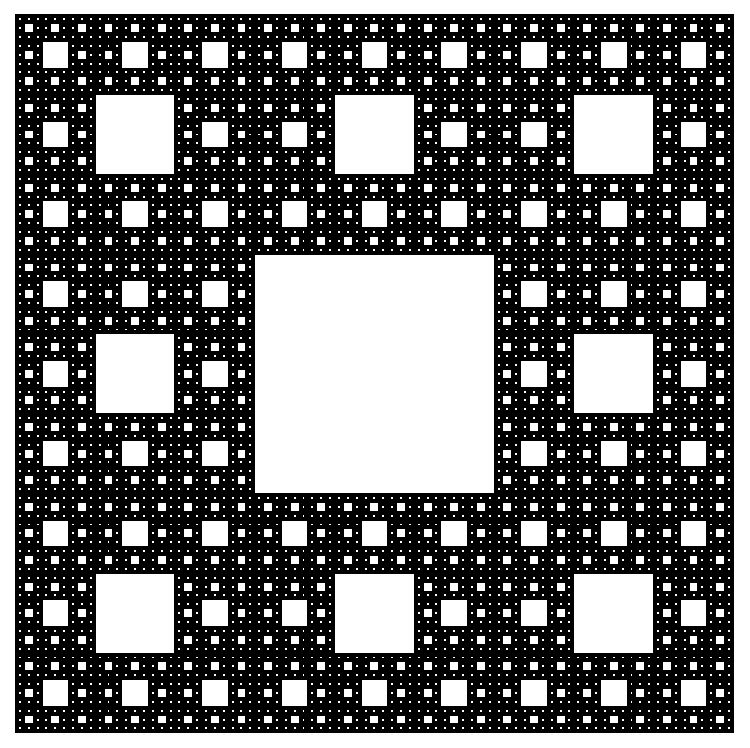}
	\caption{The standard Sierpi\'nski carpet.}\label{SC}
\end{figure}

The main difference from the Sierpi\'nski gasket is that $\mathcal{SC}$ are
infinitely ramified, i.e., $K$ cannot be disconnected by removing a finite number of
points. Hence, unlike on the Sierpi\'nski gasket, where the trace of the diffusion is the simple random walk, on a $\mathcal{SC}$, the trace of the diffusion can be complicated. Fortunately, the diffusions on $\mathcal{SC}$, along with the Markov processes on approximating domains or graphs, have been deeply studied over the past thirty years \cite{BB1,BB2,BB3,BB4,BB5,BBKT,BCK,CQ,HK,KZ}, so we have enough tools to show the existence of the scaling limits of LERW on $\mathcal{SC}$ graphs. \vspace{0.15cm}

\noindent\textbf{\emph{Sierpi\'nski carpet graphs ($\mathcal{SC}$ graphs)}.} For convenience, we let $\Theta_m=\big\{\theta=\theta_1\theta_2\cdots\theta_m:\theta_l\in\{1,2,\cdots,N\},\forall 1\leq l\leq m\big\}$ be the set of words of length $m$, and we write $\Psi_\theta=\Psi_{\theta_1}\circ\Psi_{\theta_2}\circ\cdots\circ\Psi_{\theta_n}$ for each $\theta=\theta_1\theta_2\cdots \theta_m\in \Theta_m$. Also set $\Theta_0=\emptyset$ and let $\Psi_\emptyset$ be the identity map. Let $V_0=\{(i,j):i,j=0,1\}$ be the four vertices of the unit square $H_0$. Then, for $m\geq 0$, we define the $\mathcal{SC}$ graphs as the graphs $G_m=(V_m,E_m)$, where   
\[
\begin{cases}
V_m=\bigcup_{\theta\in \Theta_m}\Psi_\theta(V_0),\\
E_m=\big\{\{x,y\}\in V_m^2:d(x,y)=k^{-m}\big\}. 
\end{cases}
\] 
See Figure \ref{SCgraphs} for the $\mathcal{SC}$ graphs associated with the standard $\mathcal{SC}$. 

\begin{figure}[htp]
	\includegraphics[width=3.5cm]{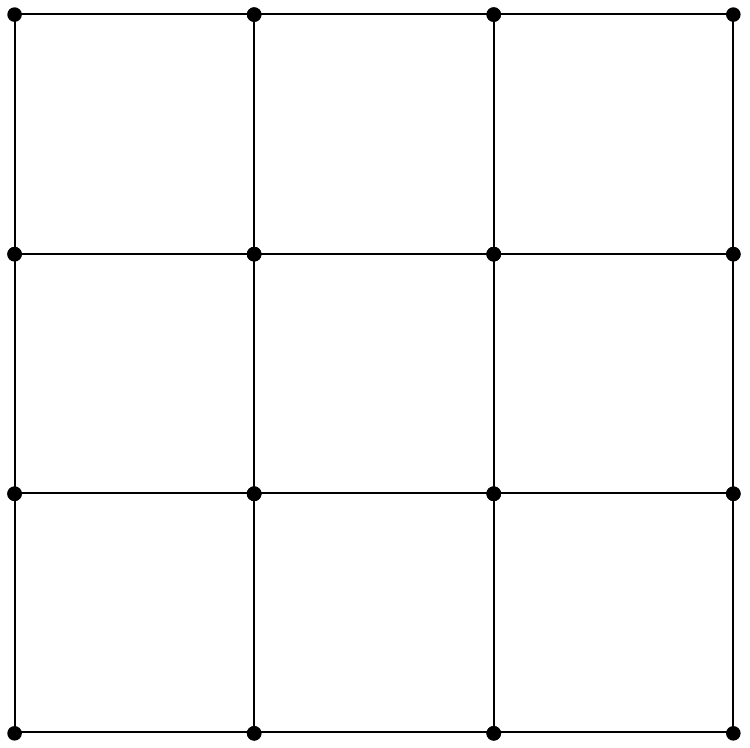}\ 
	\includegraphics[width=3.5cm]{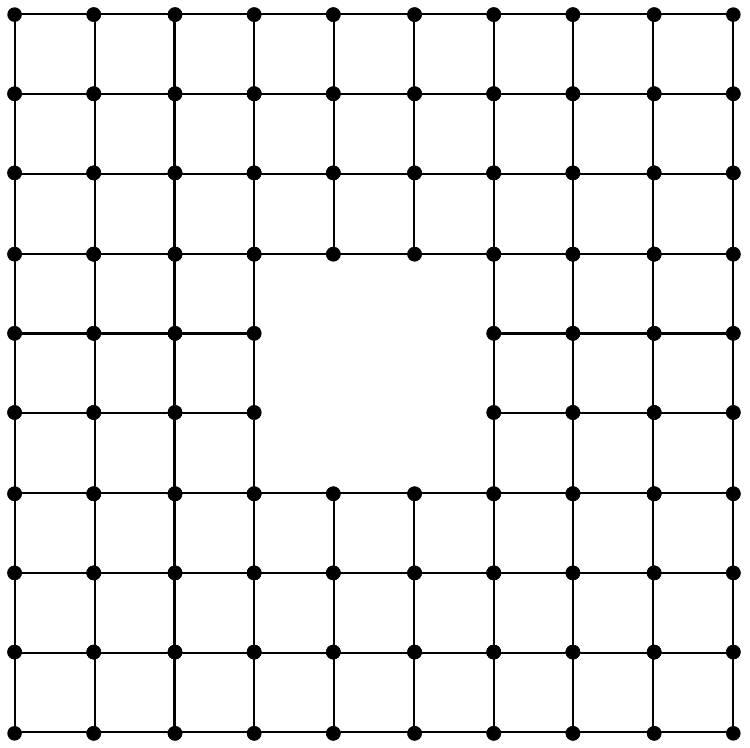}\
	\includegraphics[width=3.5cm]{SC3.pdf}
	\caption{$\mathcal{SC}$ graphs $G_1,G_2,G_3$.}\label{SCgraphs}
\end{figure}

\noindent\textbf{\emph{Resistance forms on $\mathcal{SC}$ graphs}.} On the $\mathcal{SC}$ graph $G_m$, one naturally considers the (discrete) resistance form 
\[\mcD_m(f,f)=\sum_{\{x,y\}\in E_m}\big(f(x)-f(y)\big)^2,\quad\forall f\in l(V_m),\]
which induces the simple random walk on $G_m$. We write $R_m$ for the resistance metric associated with the form $\big(\mcD_m,l(V_m)\big)$. For convenience, we also simply call $\mcD_m$ the resistance form on $G_m$. \vspace{0.15cm}

The resistance estimates have been playing a central role in the study of the diffusions on $\mathcal{SC}$. There are now two ways of proving the theorem: the first approach was developed by Barlow and Bass, which involves a face to face resistance estimate \cite{BB2}, and a proof of Harnack inequalities \cite{BB1,BB4}; the second approach was initiated by Kusuoka and Zhou \cite{KZ}, who introduced a study of Poincar\'e constants, and the last piece of the story, which makes the method fully analytic, was recently filled by the author and Qiu \cite{CQ}. Both methods have its own advantage: the method of Barlow-Bass can be extended to higher dimensional (generalized) $\mathcal{SC}$ \cite{BB4}, while the method of Kusuoka-Zhou achieved success on some irrationally ramified fractals \cite{CQ}. 

The above papers were not written for the $\mathcal{SC}$ graphs (although the methods there can be easily adjusted for $\mathcal{SC}$ graphs), see \cite{BB5,BCK} for a study of $\mathcal{SC}$ graphs.   

\begin{theorem}[\cite{BB5,BCK}]\label{thm51}
	Let $K$ be a $\mathcal{SC}$ and $G_m=(V_m,E_m),m\geq 1$ be the corresponding $\mathcal{SC}$ graphs. Then there exists a unique positive index $\gamma$ and some positive constants $C_1,C_2$ independent of $m$ such that 
	\[
	C_1\cdot\rho^\gamma\leq k^{-m\gamma}\cdot R_m\big(x,V_m\setminus B_d(x,\rho)\big)\leq k^{-m\gamma}\cdot R_m(x,y)\leq C_2\cdot\rho^\gamma,
	\]
	for any $x\in V_m$, any $\rho\geq k^{-m}$ such that $V_m\setminus B_d(x,\rho)\neq \emptyset$, and any $y\in V_m$ such that $d(x,y)=\rho$.   
\end{theorem}

The following is the celebrated uniqueness theorem \cite{BBKT} by Barlow, Bass, Kumagai and Teplyaev. Here we state a weaker version since we only need to take care of resistance forms (for planar carpets).\vspace{0.1cm}

\noindent\textbf{\emph{Unfolding mapping on $\mathcal{SC}$ \cite{BBKT}}.} Let $K$ be a $\mathcal{SC}$. For each $m\geq 0$ and $\theta\in \Theta_m$, we let the folding map $\widetilde{\Gamma}_\theta:H_0\to \Psi_\theta(H_0)$ be the unique continuous map such that $\widetilde{\Gamma}_\theta|_{\Psi_\theta(H_0)}$ is the identity mapping and $\widetilde{\Gamma}_\theta|_{\Psi_{\theta'}(H_0)}$ is an isometry $\Psi_{\theta'}(H_0)\to \Psi_\theta(H_0)$ for each $\theta'\in \Theta_m$. 

The unfolding mapping $\mathcal{U}_\theta:C(\Psi_\theta (K))\to C(K)$ is defined as $\mathcal{U}_\theta f=\mathcal{U}_\theta(f)=f\circ \widetilde{\Gamma}_\theta$ for each $f\in C(\Psi_\theta (K))$.

\noindent\textbf{\emph{Locally symmetric resistance forms on $\mathcal{SC}$ \cite{BBKT}}.} Let $K$ be a $\mathcal{SC}$, and $(\mcE,\mcF)$ be a local regular resistance form on $K$, and for convenience, we assume $\mcF\subset C(K)$. We say $(\mcE,\mcF)$ is locally symmetric if the following (1),(2),(3) hold for any $m\geq 0$, $\theta\in \Theta_m$:

(1). $\mathcal{U}_\theta(f|_{\Psi_\theta K})\in \mcF$ if $f\in \mcF$. 

(2). $\mcE(\mathcal{U}_\theta f,\mathcal{U}_\theta f)=\mcE\big(\mathcal{U}_{\theta'} (f\circ \Gamma),\mathcal{U}_{\theta'} (f\circ \Gamma)\big)$ for any $f\in C(\Psi_\theta K)$, $\theta'\in \Theta_m$ and $\Gamma$ is an isometry $\Psi_{\theta'}(H_0)\to \Psi_\theta(H_0)$. 

(3). $\mcE(f,f)=N^{-m}\sum_{\theta'\in \Theta_m}\mcE\big(\mathcal{U}_{\theta'}(f|_{\Psi_{\theta'K}}),\mathcal{U}_{\theta'}(f|_{\Psi_{\theta'K}})\big)$ for any $f\in \mcF$. \vspace{0.15cm}

\noindent\textbf{A remark about quadratic forms}. In the above, we admit the setting $\mcE(f,f)=\infty$ if $f\notin \mcF$ for any symmetric bilinear form $(\mcE,\mcF)$. We view $\mcE$ as a quadratic form on $C(K)$ with extended real values induced from $(\mcE,\mcF)$.

Conversely, given a non-negative quadratic form $\mcE$ on $C(A)$, we let $\mcF=\{f\in C(A):\mcE(f,f)<\infty\}$, and hence get an associated bilinear form $(\mcE,\mcF)$. 

To conclude, there is a natural one to one correspondence between non-negative quadratic forms and bilinear forms.\vspace{0.15cm}

The following is the celebrated theorem about the uniqueness of Brownian motions on a $\mathcal{SC}$ \cite{BBKT}. 

\begin{theorem}\cite{BBKT}\label{thm52}
Let $K$ be a $\mathcal{SC}$. Then, there exists a unique (up to scalar multiples) locally symmetric, local, regular resistance form on $K$.
\end{theorem}

By Theorem \ref{thm52}, we have the following result concerning the resistance metrics. We leave the proof to the end of the paper. 

\begin{theorem}\label{thm53}
Let $K$ be a $\mathcal{SC}$. Let $(\mcE,\mcF)$ be the unique locally symmetric, local, regular resistance form on $K$, and let $R$ be the corresponding resistance metric. Then there is a sequence of renormalization constants $c_mk^{-\gamma m}$, where $\gamma$ is the same as in Theorem \ref{thm51} and $c_m,m\geq 1$ satisfies $C^{-1}\leq c_m\leq C,\forall m\geq 1$ for some positive constant $C$ depending only on $K$, so that  
\[
R(x,y)=\lim\limits_{m\to\infty}c_mk^{-\gamma m}\cdot R_m(x_m,y_m), 
\]
for any $x,y\in K$ and $x_m,y_m\in V_m$ such that $x_m\to x,y_m\to y$. 
\end{theorem}

We have the following corollary of Theorem \ref{thm53}.

\begin{definition}\label{def54}
(a). Let $(\Omega^{(m')},X^{(m')}_n,\mathbb{P}_x^{(m')})$ be the simple random walk on $G_{m'}$, where $m'\geq 0$. For $0\leq m\leq m'$, we define
\[
X_n^{(m',m)}=X^{(m)}_{t_n},\qquad \forall n\geq 0,
\]
where $t_0=\min\{n\geq 0:X^{(m')}_n\in V_m\}$ and $t_n=\min\big\{n'> t_{n-1}:X^{(m')}_{n'}\in V_m\setminus \{X^{(m')}_{t_{n-1}}\}\big\}$ for $n\geq 1$. Then, $(\Omega^{(m')},X_n^{(m',m)},\mathbb{P}_x^{(m')})$ is a reversible random walk (Markov chain) on $V_m$.

(b). Let $(\Omega,X_n,\mathbb{P}_x)$ be the diffusion on $K$ (associated with $(\mcE,\mcF)$ and $\mu$, where $(\mcE,\mcF)$ is the unique locally symmetric, local, regular resistance form on $K$, and $\mu$ is the normalized Hausdorff measure). For $m\geq 0$, let 
\[X_n^{(*,m)}=X_{t_n},\qquad \forall n\geq 0,\]
where $t_0=\inf\{t\geq 0:X_t\in V_m\}$ and $t_n=\inf\big\{t>t_{n-1}:X_t\in V_m\setminus \{X_{t_{n-1}}\}\big\}$ for $n\geq 1$. $(\Omega,X_n^{(*,m)},\mathbb{P}_x)$ is a reversible random walk on $V_m$.
\end{definition}

\begin{corollary}\label{coro55}
$\mathbb{P}^{(m')}_x(X^{(m',m)}|_{[0,\tau_y]}\in \bullet)\Rightarrow \mathbb{P}_x(X^{(*,m)}|_{[0,\tau_y]}\in \bullet)$ as $m'\to\infty$ for each $x,y\in V_m$.
\end{corollary}

\noindent\textbf{Remark.} In the statement of Corollary \ref{coro55}, $\tau_y$ depends on the process that we talk about. 

\noindent\textbf{Remark.} $X^{(*,m)}|_{[0,\tau_y]}$ is the same as $X^{[V_m,y]}|_{[0,\tau_y]}$, which is defined in Definition \ref{def33}.

\begin{proof}
	First, we recall an easy observation about resistance forms on finite sets. Let $V$ be a finite set, and let $\mcD'_i,i\geq 1$ and $\mcD'$ be resistance forms on $V$. We let $c'_{x,y,i}\geq 0$ be the conductance between $x\neq y\in V$ associated with $\mcD'_i$, and let $c'_{x,y}\geq 0$ be the conductance between $x\neq y\in V$ associated with $\mcD'$, i.e.  $\mcD'_i(f,f)=\frac{1}{2}\sum_{x\neq y}c'_{x,y,i}\big(f(x)-f(y)\big)^2$ and $\mcD'(f,f)=\frac{1}{2}\sum_{x\neq y}c'_{x,y}\big(f(x)-f(y)\big)^2$ for each $f\in l(V)$. Then, $c_{x,y,i}\to c_{x,y},\forall x\neq y\in V$ if and only $R'_i(x,y)\to R'(x,y),\forall x\neq y\in V$ (see Lemma 2.8 of \cite{C}). Moreover, clearly, the transition kernels of the associated random walks converge if $c_{x,y,i}\to c_{x,y}$ for any $x\neq y\in V$. 
	
	Now, we notice that $(\Omega^{(m')},X_n^{(m',m)},\mathbb{P}_x^{(m')})$ is the random walk associated with the resistance metric $R_{m'}$ on $V_m$; $(\Omega,X_n^{(*,m)},\mathbb{P}_x)$ is the random walk associated with the resistance metric $R$ on $V_m$. In fact, for $a,b\in V_m$, $\mathbb{P}^{(m')}_a(X^{(m',m)}_1=b)=h(a)$, where $h$ is the unique function on $V_{m'}$ such that $h(b)=1,h|_{V_m\setminus \{a,b\}}=0$ and $h$ is harmonic elsewhere on $V_{m'}$; similarly, $\mathbb{P}_a(X^{(*,m)}_1=b)=h(a)$, where $h$ is the unique function on $K$ such that $h(b)=1,h|_{V_m\setminus \{a,b\}}=0$ and $h$ is harmonic elsewhere on $K$. In both cases, to compute $h(a)$, we only need the trace of $\mcD_{m'}$ or $(\mcE,\mcF)$ onto $V_m$, so the claim holds.
	
	The corollary follows immediately from Theorem \ref{thm53}, with the above observations. 
\end{proof}

Finally, we point out that all the proofs in Section 4.2 also work here, since we have the uniform resistance estimates (Theorem \ref{thm51}).  \vspace{0.15cm}

\noindent\textbf{$(m',\rho,V)$-Long-jumps.} Let $V$ be a finite subset of $V_{m'}$.  A $(m',\rho,V)$-long-jump in a finite path $\bw$ on $V_{m'}$ is a subpath $\bw|_{[s_1,s_2]}$ such that $d(w_{s_1},w_{s_2})\geq \rho$ and $\Im(\bw|_{[s_1,s_2]})\cap V=\emptyset$. Let $\mathscr{J}(m',\rho,V)$ denote the set of paths that have a $(m',\rho,V)$-long-jump.
 
\noindent\textbf{\emph{$(m',\rho)$-Loops}.} A $(m',\rho)$-loop is a is a subarc $\bw|_{[s_1,s_2]}$ of $\bw$ such that $w_{s_1}=w_{s_2}$ and $\Im(\bw|_{[s_1,s_2]})\nsubseteq B_d(w_{s_1},\rho)$. Let $\mathscr{L}(m',\rho)$ denote the set of paths that have a $(m',\rho)$-loop.

\begin{lemma}\label{lemma56}
Let $x,y\in V_*$ and $\rho>0$.
 	
(a). $\lim\limits_{m\to\infty}\sup\limits_{m'\geq m}\mathbb{P}^{(m')}_x\big(\mfL_{V_m}(X^{(m')}|_{[0,\tau_y]})\in\mathscr{J}(m',\rho,V_m)\big)=0$. 
 	
(b). $\lim\limits_{m\to\infty}\sup\limits_{m'\geq m}\mathbb{P}^{(m')}_x\big(\mfL_{V_m}(X^{(m')}|_{[0,\tau_y]})\in\mathscr{L}(m',\rho)\big)=0$. 
\end{lemma}

The proof is exactly the same as that of Lemma \ref{lemma46}, noticing that all the estimates in the proof of Lemma \ref{lemma46} depends on the choice of covering, and the resistance estimates, which are both uniform on $V_{m'}$ (in fact, we only need to consider large enough $m'$). Combining with Lemma \ref{lemma46}, we arrive at the last main theorem of the paper, the existence of the scaling limit of LERW that starts at $x$ and ends at $y$ on $\mathcal{SC}$.

%\begin{theorem}\label{thm57}
%Let $K$ be a $\mathcal{SC}$, let $x,y\in V_*$, and let $\nu_{x,y}$ be the distribution of loop-erased random paths defined in Theorem \ref{th2}, which starts at $x$ and ends at $y$. Then 
%\[\mathbb{P}_x^{(m')}\big(\Im\mfL(X^{(m')}|_{[0,\tau_y]})\in \bullet\big)\Rightarrow \nu_{x,y},\quad \text{ as }m'\to\infty,\]  
%where the weak convergence is on $(\mathcal{K},d_H)$, the space of compact subsets of $K$ endowed with the Hausdorff metric.
%\end{theorem}
\begin{proof}[Proof of Theorem \ref{th3}]
Let 
\[
\begin{cases}
	\nu^{(m')}_{x,y}=\mathbb{P}_x^{(m')}\big(\Im\mfL(X^{(m')}|_{[0,\tau_y]})\in \bullet\big),\\
	\tilde{\nu}^{(m',m)}_{x,y}=\mathbb{P}_x^{(m')}\big(\Im\mfL_{V_m}(X^{(m')}|_{[0,\tau_y]})\in \bullet\big),\\
	\nu^{(m',m)}_{x,y}=\mathbb{P}_x^{(m')}\big(\Im\mfL(X^{(m',m)}|_{[0,\tau_y]})\in \bullet\big),\\
	\nu^{(*,m)}_{x,y}=\mathbb{P}_x\big(\Im\mfL(X^{(*,m)}|_{[0,\tau_y]})\in \bullet\big).
\end{cases}
\]
By Lemma \ref{lemma56} (b) and Theorem \ref{th1}, we have $\lim\limits_{m\to\infty}\sup\limits_{m'\geq m}\pi(\nu^{(m')}_{x,y},\tilde{\nu}^{(m',m)}_{x,y})=0$; by Lemma \ref{lemma56} (a), we have $\lim\limits_{m\to\infty}\sup\limits_{m'\geq m}\pi(\tilde{\nu}^{(m',m)}_{x,y},\nu^{(m',m)}_{x,y})=0$; by Corollary \ref{coro55}, we have $\lim\limits_{m'\to\infty}\pi(\nu^{(m',m)}_{x,y},\nu^{(*,m)}_{x,y})=0$ for each $m\geq 0$ such that $x,y\in V_m$; lastly, by Theorem \ref{th2}, $\lim\limits_{m\to\infty}\pi(\nu^{(*,m)}_{x,y},\nu_{x,y})=0$. 
\end{proof}

\subsection{Appendix of Section 5: Proof of Theorem \ref{thm53}.} 
We will use some results from \cite{C}. Most of them are easy, and we only review a result about some non-standard $\Gamma$-convergence here. Also see \cite{C1},\cite{CHK} for related (earlier works) about the convergence of the associated stochastic processes.

\begin{definition}[\cite{C}]\label{def58}
Let $(B,d)$ be some compact metric space; let $A_n,n\geq 1$ and $A$ be compact subsets of $(B,d)$ such that $d_H(A_n,A)\to 0$ as $n\to \infty$.

(a). Let $f_n\in C(A_n)$ for $n\geq 1$ and $f\in C(A)$. We write $f_n\rightarrowtail f$ if $f(x)=\lim\limits_{n\to\infty}f_n(x_n)$ for any $x\in A$ and $x_n\in A_n,n\geq 1$ such that $x_n\to x$ as $n\to\infty$.

(b). For each $n\geq 1$, let $\mcE_n$ be a quadratic form (with extended real values)  on $C(A_n)$; let $\mcE$ be a quadratic form (with extended real values) on $C(A)$. We say $\mcE_n$ $\Gamma$-converges to $\mcE$ on $C(B)$ if and only if (a),(b) hold:

\noindent(1). If $f_n\rightarrowtail f$, where $f_n\in C(A_n),\forall n\geq 1$ and $f\in C(A)$, then 
\[\mcE(f,f)\leq \liminf_{n\to\infty} \mcE_n(f_n,f_n).\]

\noindent(2). For each $f\in C(A)$, there exists a sequence $f_n\in C(A_n), n\geq 1$ such that $f_n\rightarrowtail f$ and 
\[\mcE(f,f)=\lim_{n\to\infty} \mcE_n(f_n,f_n).\]
\end{definition}

\noindent\textbf{Remark}. It is not hard to show (see Proposition 2.3 of \cite{C}) $f_n\rightarrowtail f$ if and only if there exists $g_n\in C(B),n\geq 1$ and $g\in C(B)$ such that $f_n=g_n|_{A_n},n\geq 1$, $f=g|_A$ and $g_n$ converges uniformly to $g$. This explains the name `$\Gamma$-converge on $C(B)$'.

\noindent\textbf{Remark}. Theorem \ref{thm53} says $c_mk^{-m\gamma}R_m\rightarrowtail R$.

\begin{lemma}[\cite{C}]\label{lemma59}
Assume the same settings of Definition \ref{def58}. Let $R_n\in C(A_n^2)$ be a resistance metric on $A_n$ for $n\geq 1$, let $R\in C(A^2)$ be a resistance metric on $A$, and assume $R_n\rightarrowtail R$. Let $(\mcE_n,\mcF_n)$ be the resistance form associated with $R_n$ for $n\geq 1$, and let $(\mcE,\mcF)$ be the resistance form associated with $R$. Then, we have $\mcE_n$ $\Gamma$-converges to $\mcE$ on $C(B)$.
\end{lemma}

\begin{proof}[Proof of Theorem \ref{thm53}]
Fix $p\neq q\in V_0$, and let $c_m=k^{\gamma m}R(p,q)/R_m(p,q)$ for each $m\geq 0$. Then, by Theorem \ref{thm51}, we know that $C^{-1}<c_m<C$ for some constant $C>1$ depending only on $K$. For short, we write $R'_m(x,y)=c_mk^{-\gamma m}R_m(x,y)$ for the renormalized resistance metric on $V_m$, and $\mcD'_m=c^{-1}_mk^{\gamma m}\mcD_m$ for the renormalized resistance form on $V_m$. In particular, $R'_m(p,q)=R(p,q)$.
	
By Lemma 2.2 of \cite{C}, Theorem 2.9 of \cite{C} and Theorem \ref{thm51}, we know that there is a subsequence $m_i,i\geq 1$ and a resistance metric $R'\in C(K^2)$ such that $R_{m_i}\rightarrowtail R'$ and we let $(\mcE',\mcF')$ be the associated resistance form on $K$. We will show that $(\mcE',\mcF')=(\mcE,\mcF)$ by using the uniqueness Theorem (Theorem \ref{thm51}). We need to verify the local property, the regular property, and (1),(2),(3) of the definition of the locally symmetric forms.\vspace{0.15cm}
	
\noindent\textbf{Regular.} It is easy to see that $R'$ satisfies the estimate $C_1d(x,y)^{\gamma}\leq R'(x,y)\leq  C_2d(x,y)^{\gamma}$ by Theorem \ref{thm51} (see also Theorem 2.9 of \cite{C}). Hence, $(K,R')$ is compact, and the regular property follows from Corollary 6.4 of \cite{ki5}. \vspace{0.15cm}

\noindent\textbf{Local.} Let $f,g\in \mcF'$ such that $\text{supp}(f)\cap \text{supp}(g)=\emptyset$, so that $\inf\{d(x,y):x\in \text{supp}(f),y\in \text{supp}(g)\}>0$. We can find $f_{m_i}\in l(V_{m_i}),g_{m_i}\in l(V_{m_i}),i\geq 1$ such that $f_{m_i}\rightarrowtail f,g_{m_i}\rightarrowtail g$ and $\mcE'(f,f)=\lim\limits_{i\to\infty}\mcD'_{m_i}(f_{m_i},f_{m_i}),\mcE'(g,g)=\lim\limits_{i\to\infty}\mcD'_{m_i}(g_{m_i},g_{m_i})$ by Lemma \ref{lemma59}. In addition, by the remark below Definition \ref{def58}, by Lemma \ref{lemma59}, and by an easy application of the Markov property, we can in addition assume that 
\[\big(f_{m,i}(x)-f_{m,i}(y)\big)\big(g_{m,i}(x)-g_{m,i}(y)\big)=0,\quad\forall\{x,y\}\in E_{m_i},\]
for any $i$ large enough. Hence, by Lemma \ref{lemma59},
\begin{align*}
\mcE'(f+g,f+g)&\leq \liminf\limits_{i\to\infty} \mcD'_{m_i}(f_{m,i}+g_{m,i},f_{m,i}+g_{m,i})\\
&=\mcE'(f,f)+\mcE'(g,g)\\
\mcE'(f-g,f-g)&\leq \liminf\limits_{i\to\infty} \mcD'_{m_i}(f_{m,i}-g_{m,i},f_{m,i}-g_{m,i})\\
&=\mcE'(f,f)+\mcE'(g,g)
\end{align*}
This implies that $\mcE'(f+g,f+g)=\mcE'(f-g,f-g)=\mcE'(f,f)+\mcE'(g,g)$, noticing that $\mcE'(f+g,f+g)+\mcE'(f-g,f-g)=2\mcE'(f,f)+2\mcE'(g,g)$. Hence $\mcE'(f,g)=0$.\vspace{0.15cm}
	
\noindent(1). For any $f\in \mcF'$, we can find $f_{m_i}\in l(V_{m_i})$ such that $f_{m_i}\rightarrowtail f$ and $\mcE'(f,f)=\lim\limits_{i\to\infty}\mcD'_{m_i}(f_{m_i},f_{m_i})$ by Lemma \ref{lemma59}. Then, by Lemma \ref{lemma59},
\begin{align*}
&\mcE'\big(\mathcal{U}_{\theta}(f|_{\Psi_\theta K}),\mathcal{U}_{\theta}(f|_{\Psi_\theta (K)})\big)\\
\leq&\liminf\limits_{i\to\infty}\mcD'_{m_i}\big((f_{m_i}|_{\Psi_\theta (V_{m_i-m})})\circ\widetilde{\Gamma}_\theta,(f_{m_i}|_{\Psi_\theta (V_{m_i-m})})\circ\widetilde{\Gamma}_\theta\big)\\
\leq& N^{m_i}\liminf\limits_{i\to\infty}\mcD'_{m_i}(f_{m_i},f_{m_i})
=N^{m_i}\mcE'(f,f).
\end{align*}
Hence $\mathcal{U}_{\theta}(f|_{\Psi_\theta K})\in \mcF'$. \vspace{0.15cm}
	
By a same argument as in (1), one can show that for any $m\geq 0$, $\theta,\theta'\in \Theta_m$ and any isometry $\Gamma:\Psi_{\theta'}(H_0)\to \Psi_{\theta}(H_0)$, if $\mathcal{U}_\theta(f)\in \mcF'$, where $f\in C\big(\Psi_\theta (K)\big)$, then $\mathcal{U}_{\theta'}(f\circ \Gamma)\in \mcF'$. So in the following, we do not need to worry about whether $f\in \mcF'$, and it remains to show the equities in (2),(3). We introduce the tool of energy measures. \vspace{0.15cm}
	
\noindent\textbf{\emph{Energy measure}.} The energy measure $\nu_f$ associated with $f\in \mcF'$ (and $(\mcE',\mcF')$) is the unique Radon measure on $K$ such that
\[\int_K g(x)\nu_f(dx)=2\mcE'(fg,f)-\mcE'(f^2,g), \quad \forall g\in\mcF'.\]

It is well known that $\nu_f(K)=2\mcE'(f,f)$ (see Lemma 3.2.3 of \cite{FOT}) as $(\mcE',\mcF')$ is strongly local. A useful fact is that $\nu_f(A)$, where $A$ is a Borel subset of $K$, only depends on the value of $f$ on $A$ (this is an observation in \cite{BBKT}. See Lemma 2.7 of \cite{BBKT} and page 123 of \cite{FOT} for essential tools).\vspace{0.15cm}
	
\noindent\textbf{Claim 1.} Let $\square_1=\Psi_{\theta}H_0$ and $\square_2=\Psi_{\theta'}H_0$ for some $\theta,\theta'\in \Theta_m,m\geq 0$, and let $\Gamma:\square_2\to\square_1$ be an isometry. Let $A_2$ be any Borel subset of $\text{int}(\square_2)\cap K$, and $A_1=\Gamma(A_2)$. Then, if $f_1,f_2\in \mcF$ and $f_2|_{\Psi_{\theta'}(K)}=f_1|_{\Psi_\theta(K)}\circ\Gamma$, we have $\nu_{f_1}(A_1)=\nu_{f_2}(A_2)$. 
	
\begin{proof}[Proof of Claim 1]
We can see that for any $f\in \mcF'$ supported on $\square_1\cap K$, 
\[\mcE'(f,f)=\mcE'(f\circ\Gamma,f\circ \Gamma),\]
where with a little abuse of notation, $f\circ \Gamma(x)=f\big(\Gamma(x)\big)$ if $x\in K\cap\square_2$, and $f\circ \Gamma|_{K\setminus \square_2}=0$. To see this, we apply Lemma \ref{lemma59} and the Markov property to see that there is a sequence $f_{m_i}\in l(V_{m_i}),i\geq 0$, such that $f_{m_i}(x)=0,\forall x\in \Psi_{\theta'}(V_{m_i-m})$ for any $m_i\geq m$ and $\theta'\in \Theta_m\setminus \{\theta\}$, and $\mcE'(f,f)=\lim\limits_{i\to\infty}\mcD'_{m_i}(f_{m_i},f_{m_i})$.  Again, by Lemma \ref{lemma59},
\[
\begin{aligned}
	\mcE'(f,f)&=\lim\limits_{i\to\infty}\mcD'_{m_i}(f_{m_i},f_{m_i})\\
	&=\lim\limits_{i\to\infty}\mcD'_{m_i}(f_{m_i}\circ \Gamma,f_{m_i}\circ \Gamma)\geq \mcE'(f\circ\Gamma,f\circ\Gamma).
\end{aligned}
\]
The reverse direction, $\mcE'(f,f)\leq \mcE'(f\circ\Gamma,f\circ\Gamma)$, can be proved with a same argument. 

Next, we choose compact $A_2'\subset A_2$ (so $d(A_2',\partial\square_2)>0$) and let $A_1'=\Gamma(A_2')$. Then we choose $f\in \mcF'$ supported on $\square_1\cap K$ (by using the regular property) such that $f|_{A_1'}=f_1|_{A_1'}$, so $f\circ \Gamma|_{A_2'}=f_2|_{A_2'}$. One can see that 
\[
\begin{aligned}
&\int_K g(x)\nu_f(dx)
=2\mcE'(fg,f)-\mcE'(f^2,g)\\
=&2\mcE'\big((fg)\circ \Gamma,f\circ \Gamma\big)-\mcE'\big((f^2)\circ\Gamma,g\circ\Gamma\big)
=\int_K g\circ \Gamma(x)\nu_{f\circ\Gamma}(dx)
\end{aligned}
\]
for any $g\in \mcF'$ supported on $\square_1\cap K$. One can check that $\{g\in \mcF':g|_{K\setminus\square_1}=0\}$ is dense in $\{g\in C(K):g|_{K\setminus \square_1}=0\}$ by the regular property of $(\mcE',\mcF')$. It follows that $\nu_{f_1}(A'_1)=\nu_f(A'_1)=\nu_{f\circ\Gamma }(A'_2)=\nu_{f_2}(A'_2)$.

Finally, by the inner regular property of the energy measure (as a Radon measure),  we have $\nu_{f_1}(A_1)=\nu_{f_2}(A_2)$.
\end{proof}

\noindent\textbf{Claim 2.} Let $\square=\Psi_{\theta}H_0$ for some $\theta\in \Theta_m,m\geq 0$. Then, $\nu_f(\partial\square)=0,\forall f\in \mcF$.  

\begin{proof}[Proof of Claim 2]  Since $R'$ satisfies the estimate $C_1d(x,y)^{\gamma}\leq R'(x,y)\leq  C_2d(x,y)^{\gamma}$ and the normalized Hausdorff measure $\mu$ on $(K,d)$ is Ahlfors regular, by Theorem 15.10 and 15.11 of \cite{ki5}, the heat kernel associated with $(\mcE',\mcF')$ on $L^2(K,\mu)$ has the sub-Gaussian estimates. Hence $(\mcE',\mcF')$ is comparable with the standard form $(\mcE,\mcF)$, which also admits the sub-Gaussian heat kernel estimates \cite{BB2,BB4}, in the sense that $\mcF'=\mcF$ and $C_3\mcE\leq \mcE'\leq C_4\mcE$ for some constants $C_3,C_4$ by Theorem 4.2 of \cite{GHL}. 

Hence, the claim follows by the domination principle of energy measures (see page 389 of \cite{Mosco}), and by Proposition 3.8 of \cite{HK} (which imples that energy measure on the boundary of squares associated with any $f\in \mathcal{F}$ and $(\mcE,\mcF)$ is $0$).
\end{proof}
	
The equations in (2),(3) of the definition can be easily verified with Claim 1,2. Hence $R'=R$. Noticing that the argument actually works for any subsequence, the theorem follows.
\end{proof}

\section*{Acknowledgments.}
The author is grateful for Professor Robert S. Strichartz for his long time support. He also wants to thank Professor David A. Croydon for his helpful suggestions about the writing.

%\subsection*{Conflicts of interest} The Authors declare that there is no conflict of interest.

%\subsection*{Data availability statement.} Data sharing not applicable to this article as no datasets were generated or analysed during the current study.

\bibliographystyle{amsplain}

%\clearpage
%\newpage
\end{document}